\documentclass{amsart}[11pt]
\usepackage{amssymb,amsmath,color,hyperref}
\usepackage[foot]{amsaddr}
\usepackage{color}
\usepackage{tikz}
\usepackage{pgf}
\usepackage{bm}
\usepackage[left=1.2in,right=1.2in,top=1.2in,bottom=1.2in]{geometry}
\usetikzlibrary{calc, positioning, arrows}
\usepackage{stmaryrd}

\usepackage{amsmath}
\usepackage{amsthm}
\usepackage{amssymb}
\usepackage{mismath}
\usepackage{geometry,mathtools}
\usepackage{verbatim}
\usepackage{todonotes}
\usepackage{enumitem}   
\usepackage{algorithm}
\usepackage{algpseudocode}

\usepackage{amsmath}
\usepackage{graphicx}
\usepackage{amsthm}
\usepackage{amssymb}
\usepackage{mismath}
\usepackage{quiver}
\newtheorem{theorem}{Theorem}
\newtheorem{corollary}{Corollary}
\newtheorem{definition}{Definition}
\newtheorem{lem}{Lemma}
\newtheorem{proposition}{Proposition}
\newtheorem{observation}{Observation}

\newtheorem{remark}{Remark}
\newtheorem{claim}{Claim}

\usepackage{amsmath}
\DeclareMathAlphabet{\mymathbb}{U}{BOONDOX-ds}{m}{n}

\title{Social Networks: Enumerating Maximal Community Patterns in $c$-Closed Graphs}

\author{Gabriela Bourla}
\thanks{University of Waterloo, ON, Canada.
Email: {\tt gbourla@uwaterloo.ca}} 
\author{Kaixin Wang}
\thanks{Duke University, NC, USA.
Email: {\tt kai.wang@duke.edu}}
\author{Fan Wei}
\thanks{Duke University, NC, USA. 
Email: {\tt fan.wei@duke.edu}}
\author{Runtian Zhou}
\thanks{Duke University, NC, USA. 
Email: {\tt daniel.zhou@duke.edu}}

\begin{document}

\maketitle

\begin{abstract}
Jacob Fox, C. Seshadhri, Tim Roughgarden, Fan Wei, and Nicole Wein \cite{paper1} introduced the model of $c$-closed graphs—a distribution-free model motivated by triadic closure, one of the most pervasive structural signatures of social networks. While enumerating maximal cliques in general graphs can take exponential time, it is known that in $c$-closed graphs, maximal cliques and maximal complete bipartite subgraphs can always be enumerated in polynomial time. These structures correspond to blow-ups of simple patterns: a single vertex or a single edge, with some vertices required to form cliques.
In this work, we explore a natural extension: we study maximal blow-ups of arbitrary finite graphs $H$ in $c$-closed graphs. We prove that for any fixed graph $H$, the number of maximal blow-ups of $H$ in an $n$-vertex $c$-closed graph is always bounded by a polynomial in $n$. We further investigate the case of induced blow-ups and provide a precise characterization of the graphs $H$ for which the number of maximal induced blow-ups is also polynomially bounded in $n$.
Finally, we study the analogue questions when $H$ ranges over an infinite family of graphs.
\end{abstract}

\section{Introduction}
Social networks provide a powerful framework for modeling real-world communities.
Over the past few decades there has been extensive theoretical research on the structural properties of social networks and the computational complexity of algorithms operating on them. A widely accepted consensus is that social networks exhibit distinctive features, such as heavy-tailed degree distribution, triadic closure, community-like structures, and the “small-world” property \cite{paper14, paper27, paper28, paper29}. These properties set social networks apart from the random graph $G(n,p)$ and from some arbitrary worst case graphs.  

\textbf{Triadic closure} is the intuitive idea that vertices sharing a common neighbor are more likely to be adjacent. In the context of friendship network, this means people with mutual friendships are more likely to be friends themselves.  
Empirical evidence supports triadic closure, as demonstrated in the analysis of email communications among Enron employees conducted by Fox et al.~\cite{paper1}. In this case, the vertices were Enron employees, and an edge represented at least one email sent between two employees. The model shows that as the number of common neighbors increases, the proportion of vertex pairs that are adjacent - among those with that many common neighbors - approaches one.

Over the years, numerous models of social networks have been proposed, including preferential attachment, the copying model, Kronecker graphs, and the Chung–Lu random graph model \cite{paper1, paper14, paper15, paper16, paper17, paper18, paper19, paper20, paper30}. 
Almost all social network models are generative and incorporate some of the  aforementioned distinctive properties characteristic of social networks. 

To systematically study the combinatorial properties of social networks exhibiting triadic closure—without relying on any specific generative model, Fox, Roughgarden, Seshadhri, Wei, and Wein introduced a new distribution-free graph model called the \textbf{$c$-closed graph}. This model is defined by the fundamental combinatorial property of social networks: triadic closure.

\begin{definition}
A graph $G = (V,E)$ is called \textbf{$c$-closed} for some integer $c \geq 1$ if any two non-adjacent vertices have less than $c$ common neighbors.
\end{definition}

In other words, any two vertices with at least $c$ common neighbors are adjacent. The extreme case occurs when $c=1$, in which the graph is a disjoint union of cliques. 
In the graph of the Enron emails, the closure rate increases as the number of common neighbors increases. (See Figure \ref{fig:enron}) Kvals{\o}ren showed that the $c$-closure of a graph can be computed in $O(nd^3+n^2d)$ time, which exploits another common property of social networks, a low degeneracy value $d$.

\begin{figure}[h]
    \centering
    \includegraphics[scale=0.75]{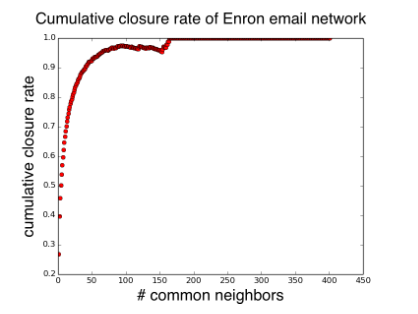}
    \caption{From Fox et al. paper, modeling triadic closure in Enron emails \cite{paper1}}
    \label{fig:enron}
\end{figure}

In the graph, the $y$-axis shows the \emph{closure rate}, which is the fraction of pairs of nodes with at least $x$ common neighbors that are themselves connected by an edge.

The goal of the program to study the new family of \(c\)-closed graphs is to understand how the additional property of \(c\)-closure influences computational complexity. While social networks serve as useful models for real-world situations, they are also expected to make certain NP-hard optimization problems easier to solve compared to worst-case graphs \cite{paper1, paper8, paper10, paper11, paper7}. Problems in P likewise benefit from this structure; for example, efficient algorithms exist for transitive closure, maximum matching, determinant computation, PageRank, and matrix inversion \cite{paper9}. Additionally, faster algorithms have been developed for computing diameter, radius, distance oracles, and identifying the most central vertices \cite{paper13}.

Community detection is a classic problem in social networks with many real-world applications, including biology and social media \cite{paper21, paper22, paper23, paper24}. It is well known that enumerating maximal cliques in arbitrary graphs is NP-hard. However, in practice, this task is often empirically fast. The concept of \(c\)-closedness provides a theoretical explanation for the gap between the worst-case complexity and the typically observed behavior in social networks, as demonstrated by the results of Fox et al. \cite{paper1}.

\begin{theorem}\label{theorem:maxcliques} \textbf{\cite{paper1}}
For positive integers $c,n$, there are at most 
\begin{equation}
    \min \{3^{(c-1)/3}n^2, 4^{(c+4)(c-1)/2}n^{2-2^{1-c}}\} 
\end{equation}
maximal cliques in any $c$-closed graph on $n$ vertices.
\end{theorem}

Their proof can also be turned into an algorithm. In particular, in \cite{paper1} it was shown that it takes $O(p(n,c) + 3^{c/3}n^2)$ time to generate a set that contains all the maximal cliques of a given $c$-closed graph, where $p(n,c)$ is the time it takes to list all 2-paths in a $c$-closed graph with $n$ vertices. The exact set of the maximal cliques can be generated in time $O(p(n,c) + 3^{c/3}2^c cn^2)$. This complexity was further improved by Manoussakis \cite{paper6}. 

Enumerating other structures besides cliques was also  extensively investigated by Koana et. al. and Lokshtanov et. al. \cite{paper2, paper3, paper32}, where they search for sparse or dense subgraphs or dominating sets.

Theorem \ref{theorem:maxcliques} \cite{paper1} investigates the enumeration of maximal cliques in a 
$c$-closed graph. Recall that a maximal clique can be seen as a maximal “blow-up” of a single vertex into a clique. In this paper, we extend this idea by considering the enumeration of maximal blow-ups of any graph 
$H$, which corresponds to identifying communities with predetermined inter- or intra-community relationships. Further details are provided in the following subsections.

\subsection{Notation and Terminology}\label{sec: notation}

All graphs in this paper are simple, undirected, and unweighted. We denote the vertex and the edge set of a graph $G$ by $V(G)$ and $E(G)$ respectively. We denote the size of vertex and edge sets by $v(G)$ and $e(G)$, respectively. The \textbf{distance} between two distinct vertices $u,v\in V(G)$, denoted by $d(u,v)$, is the length of the shortest path between $u$ and $v$. The \textbf{neighborhood} of a vertex $v\in V(G)$ is denoted by $N_G(v)$ (or $N(v)$ if there is no confusion) and is defined as all vertices adjacent to $v$ but not including $v$. For any $r> 0$, we use $N_{\leq r}(v)$ to denote the set of vertices which are distance at most $r$ from $v$ but not including $v$. We use $\deg(v)$ to denote the degree of a vertex $v$. The minimum degree of a graph $G$ is denoted by $\delta(G)$. The maximum degree of a graph $G$ is denoted by $\Delta(G)$. The size of a maximum independent set of $G$ is denoted by $\alpha(G)$.

Two vertices $v\ne v'$ are \textbf{twins} if they share exactly the same neighbors, i.e. $N(v) \setminus v' = N(v') \setminus v$. They can, but are not required to, be adjacent to each other.
Given a graph $H$, we say two vertices $v, v'$ in $H$ are \textbf{plus twins} if $v,v'$ are adjacent to each other in $H$ and they are twins. It is easy to see that being twins defines an equivalence relation on $V(H)$. For each $v\in V(H)$, the \textbf{twin group} of $v$, denoted by $[v]$, is the set containing all twins of $v$ and $v$ itself. Similarly, being plus twins defines an equivalence relation on $V(H)$. Therefore we can define a \textbf{plus twin group} as the equivalence class under the plus twins equivalence relation. For example, if a vertex $v$ has no other plus twins, then the plus twins group of $v$ is $\{v\}$. 

In $H$, we say a vertex $i$ is \textbf{bad} if its neighbors form a clique in $H$.  In particular, if a vertex has no or one neighbors, then it is a bad vertex. By a \textbf{bad plus twin group}, we mean a plus twins group consisting of bad vertices. Clearly if $v$ is a bad vertex, and $v'$ is a plus twin of $v$, then $v'$ is also a bad vertex; and thus either every vertex in the same plus twins group is bad, or none of them are bad. 

\begin{remark}
 If two vertices $w,t\in H$ are in different twin groups but have the same adjacency to the vertices in other twin groups, then none of the following can happen: 
\begin{enumerate}
\item both $[w]$, $[t]$ are cliques in $H$ and $w,t$ are adjacent;
\item both $[w]$, $[t]$ are independent sets and $w,t$ are not adjacent.
\end{enumerate}
\end{remark}

\subsection{Main Results}

By the results of Fox et al. \cite{paper1} and Koana et al. \cite{paper2}, enumerating maximal cliques and maximal (non-induced) complete bipartite graphs becomes polynomial in the size of the host graph when the host graph is $c$-closed. Both cliques and (non-induced) complete bipartite graphs can be viewed as \textbf{blow-ups} of smaller graphs. A blow-up of a graph $H$ is a graph with a vertex partition in which each part is nonempty and corresponds to a vertex of $H$, and any two parts corresponding to adjacent vertices in $H$ are complete to each other. We now give a formal definition below.

\begin{definition}\label{def:prescribedplus}
    Given a graph $H$ on $k$ vertices and a subset $U \subset V(H)$,  a {\textbf{$U$-clique prescribed blow-up of $H$}} (or a \textbf{$(H,U)$-blow-up}) is a graph $H'$ with disjoint nonempty vertex sets $V_1 \sqcup \dots \sqcup V_k$ such that if $(i,j) \in E(H)$ then $V_i$ is complete to $V_j$ (meaning for all $v_i\in V_i$ and $v_j\in V_j$, there is an edge between $v_i$ and $v_j$), and for each $i \in U$, the set $V_i$ is a clique (there are no restrictions for $V_i$ for $i\notin U$). 
    
    In this context, we call $H$ the {\it base graph}. We call each $V_i$ a ``group" corresponding to a vertex in the base graph. When the context is clear, we will say $H'$ is a blow-up of $H$ (or an $H$-blow-up), with $U$ implicit. 
\end{definition}
For example, a clique is a blow-up of a single vertex where the blown-up vertex set is a clique; thus a clique is a $\{v\}$-prescribed blow-up of a single vertex $v$; the 
complete bipartite graph with the two bipartite classes $A\cup B$ is a blow-up of a single edge $(a,b)$ where vertex $a$ is blown up to the vertex set $A$, and the vertex $b$ is blown up to the vertex set $B$. Here we do not require either $A$ or $B$ to be a clique, and thus a complete bipartite graph is an $\emptyset$-prescribed blow-up of a single edge $(a,b)$.

Let $G$ be a simple graph. Denote the vertex set of $G$ by $V(G)$. For a subset of vertices $S\subseteq V(G)$, let $G[S]$ denotes the subgraph of $G$ induced by $S$. 
\begin{definition}\label{def:maxblow-up}
Let $G$ and $H$ be simple graphs and $S\subseteq V(G)$. We call $G[S]$ a \textbf{maximal blow-up} of $H$ in $G$, if $G[S]$ is a blow-up of $H$ and there does not exist $S\subsetneq S'\subseteq V(G)$ such that $G[S']$ is a blow-up of $H$.
\end{definition}

An example of a maximal blow-up is in Figure \ref{fig:max-def}. In this paper, we will focus on counting subsets $S\subset V(G)$ such that $G[S]$ is a maximal blow-up of $H$ on $G$.

Note that in our definition of $G[S]$ as a maximal blow-up of $H$, we do not specify the vertex partition of $S$ corresponding to the vertices of $H$. This omission poses a challenge in our analysis, as it makes it difficult to determine whether $G[S]$ is truly a maximal blow-up without a given partition. The difficulty arises because, for sets $S \subsetneq S'$ where both induce blow-ups of $H$, the vertex groups in a partition of $S$ may not align cleanly with those in a partition of $S'$—that is, a group in $S$ may not be a subset of any group in $S'$. We will elaborate on this issue in Subsection~\ref{subsec:diff}.

Intuitively, in a blow-up, each group can represent a community. The blow-up structure enforces that every person in one community is friends with every person in any adjacent community, corresponding to adjacent vertices in the original graph. For nonadjacent communities, friendships may exist but are not required. Within each prescribed community, we require that all members are mutual friends. In contrast, for non-prescribed communities, no internal friendship structure is enforced.

Enumerating community patterns is an important question in social networks. $H$-blow-ups capture flexible community patterns where groups follow a template structure $H$ but allow variable group sizes, providing a more realistic model than exact subgraph matching for social network analysis.

Beyond social networks, blow-ups arise naturally in property testing, where being $\varepsilon$-far from a property often implies the existence of a large blow-up of a forbidden subgraph, likely to be detected via random sampling. From \cite{alon-graphreg}, we can see this is closely tied to the Graph Regularity Lemmas and Removal Lemmas—foundational results in combinatorics and TCS—where regular partitions can be viewed as structured or pseudorandom samples of blow-ups. In this view, the distance from a property is witnessed by counting the patterns in such blow-ups, making them essential to both the design and analysis of testing algorithms. These results also spurred the development of graph limits. However, studying blow-ups and related questions is a lot more challenging when the underlying graph is sparse. Our work studies enumeration of blow-ups through new angles and methods, especially in the sparse graph setting in the context of social networks.
\subsubsection{Our Results}
In Section~\ref{sec: noninduced finite} we prove our first main theorem: the number of maximal blow-ups of a connected graph $H$ is always polynomial in the number of vertices of a $c$-closed graph $G$, generalizing the results of \cite{paper1, paper2} where $H$ is a single vertex or an edge. 

\begin{theorem}\label{thm:noninduced}
    Let $H$ be a graph on $k \geq 2$ vertices and $U$ a subset of $V(H)$. 
    For $c > 0$, the number of sets which form maximal $U$-prescribed blow-ups of $H$  in any $c$-closed graph on $n$ vertices is bounded above by \[n^{(\max (c-1,1))k} \cdot (\min \{3^{(c-1)k/3}n^{2k}, 4^{(c+4)(c-1)k/2}n^{(2-2^{1-c})k}\}).\]
\end{theorem}

We can similarly define an ``induced" blow-up of $H$ so that there are no edges between parts corresponding to non-adjacent vertices of $H$. Here we have the option to force not only friendship but also non-friendship between different communities and within each community.

\begin{definition}\label{def:prescribedplusminus}
    Given a graph $H$ on $k$ vertices and two disjoint subsets $U_+, U_- \subset V(H)$ (possibly empty),  a {\textbf{$(U_+, U_-)$--prescribed induced blow-up of $H$}} (or a \textbf{$(H,U_+,U_-)$-blow-up}) is a graph $H'$ with disjoint vertex sets $V_1 \sqcup \dots \sqcup V_k$ such that if $(i,j) \in E(H)$ then $V_i$ is complete to $V_j$, if $(i,j) \notin E(H)$ then $V_i$ is empty to $V_j$, and in addition, for each $i \in U_+$, the set $V_i$ is a clique and for each $i \in U_-$, the set $V_i$ is an independent set. We call a vertex $v \in V(H)$ {\it prescribed} if $v \in U_+ \cup U_-$, and say it is {\it un-prescribed} if otherwise. When context is clear, we will not specify $U_+$ and $U_-$ and call $H'$ a prescribed induced blow-up of $H$. 
\end{definition}

If $U_- = U_+ = \emptyset$, then a $(\emptyset, \emptyset)$--prescribed induced blow-up of $H$ is simply an induced blow-up of $H$ where we only require friendship and non-friendship between different communities.
If a group in a blow-up is prescribed to be a clique (i.e., this group is blown-up from a vertex in $U_+$), we can say that within a community all members are friends with each other. If a group is prescribed to be an independent set (i.e., this group is blown-up from a vertex in $U_-$), then all friendships are forbidden within that group. 

\

Our second main result, presented in section~\ref{sec: induced finite}, is a full characterization of when the number of maximal induced blow-ups of $H$ is polynomial in  $c$-closed graphs. Furthermore, there is a dichotomy: either $H$'s structure allows a nice regroup within a blow-up, which makes this number polynomial for every $c$, or a combinatorial explosion of valid partitions can be found, which makes this number exponential for some $c$. The characterization theorem is below; some definitions used within the theorem statement could be found right before the statement of Theorem \ref{thm:inducedchar}. 

\begin{theorem}\label{thm:inducedchar}
    Let $H$ be a graph on $k$ vertices where $k \geq 1$. Let $U_+, U_- \subseteq V(H)$ be two disjoint subsets. 
    \begin{enumerate}
        \item Suppose in $H$ there are no bad vertices, or exactly one bad twin group $B$ (note $H[B]$ is either a clique or an independent set) and at least one of the following hold:
        
        \begin{enumerate}
            \item If $H[B]$ is a clique, and there is at least one unprescribed vertex in $B$,
            \item If $H[B]$ is a clique, and $B\subseteq U_+$. This includes the case where $|B|=1$ and the unique vertex in $B$ is in $U_+$. 
            \item If $H[B]$ is an independent set, there is exactly one unprescribed vertex in $B$, and every other vertex in $B$ is in $U_-$.
            
        \end{enumerate}
        
        Then, for any $c>0$ and in any $c$-closed graph $G$ on $n$ vertices, the number of vertex sets corresponding to maximal induced $(U_+, U_-)$--prescribed blow-ups of $H$ in $G$ is at most $(2n)^k (n^22^c)^{k}$.
        \item Suppose in $H$ either there are at least two bad twin groups; or there is exactly one bad twin group $B$ and one of the following hold:
        
        \begin{enumerate}
            \item $H[B]$ is a clique and there is no unprescribed vertex in $B$, i.e., $B\subseteq U_- \sqcup U_+$, and in addition, $B\cap U_- \ne \emptyset$.
            \item $H[B]$ is an independent set and $B\subseteq U_- \sqcup U_+$ and $B\cap U_- \ne \emptyset$.

            \item $H[B]$ is an independent set, and there exists at least two vertices in $B$ that are not in $U_-$.
            
        \end{enumerate}
        
        Let $k:=v(H)$. For any $n\ge 1000k^5$, there is a $(k+1)$-closed graph $G$ with $v(G) \in \{n,n+1\}$ such that the number of vertex sets corresponding to maximal $(U_+, U_-)$--prescribed induced blow-ups of $H$ is at least $2^{-3k^2-2k+\frac n2}$. If furthermore $H$ is $c$-closed, then $G$ can be constructed to be $(c+2)$-closed.
    \end{enumerate}
\end{theorem}

One can check that the cases above cover all the possible triples of $H, (U_+, U_-)$.  

\

Our last results, shown in section~\ref{sec: bounded} and section~\ref{sec:unbounded}, extends the base graph $H$ to a family of graphs. For example, we may be interested in the number of maximal blow-ups of cycles in a graph $G$, where we do not specify the length of the cycle.

\begin{definition}\label{def: polyexpinf}
Let $\mathcal{H} = \{(H_1,U_1), (H_2,U_2), \dots\}$ be a family of graphs, where $H_i$ is a simple graph and $U_i \subseteq V(H_i)$ is the set of vertices in $H_i$ that are prescribed to be cliques. We call each $H_i\in \mathcal{H}$ a \textbf{pattern}. 

We say $\mathcal{H}$ is a \textbf{$c$-closed polynomial pattern family}, if there is a polynomial $f(x)\in \mathbb R[x]$ such that for any arbitrary $c$-closed graph $G$ on $n$ vertices,  and for each $(H_i, U_i) \in \mathcal{H}$, $$\# \{S\subset V(G) \colon  G[S] \text{ is a maximal }  U_i\text{-clique prescribed blow-up of $H_i$}\} \le f(n).$$ 
 In other words, roughly speaking, the number of sets $S$ which correspond to a maximal $U_i$-clique prescribed blow-up of some pattern in $\mathcal{H}$ is small (polynomial). 

We say $\mathcal{H}$ is a \textbf{$c$-closed exponential pattern family}, if there is a constant $\epsilon>0$ and a sequence of $c$-closed graphs $(G_n)_{n\ge 1}$  such that $v(G_n) \to \infty$, and for each integer $n \geq 1$, there exists an $i$ such that  $$\# \{S\subset V(G_n) \colon  G_n[S] \text{ is a maximal }  U_i\text{-clique prescribed blow-up of $H_i$}\} \ge e^{v(G_n)^\epsilon}.$$  
In other words, the number of sets $S$ which correspond to a maximal $U_i$-clique prescribed blow-up of some pattern in $\mathcal{H}$ is large (exponential). 
\end{definition}

Note that this definition can be easily altered for induced blow-ups.
When we do not specify uninduced or induced, we usually mean uninduced. 

\begin{definition}\label{def: familyinduced}
Let $\mathcal{H} = \{(H_1,U_{1,+},U_{1,-}), (H_2,U_{2,+},U_{2,-}), \dots\}$ be a family of graphs, where $H_i$ is a simple graph and $U_{i,+}, U_{i,-} \subseteq V(H_i)$ are two disjoint subsets of vertices; vertices in $U_{i,+}$ are prescribed to be cliques while vertices in $U_{i,-}$ are prescribed to be independent sets. We call each $H_i\in \mathcal{H}$ a \textbf{pattern}. 

We say $\mathcal{H}$ is a \textbf{$c$-closed induced polynomial pattern family}, if there is a polynomial $f(x)\in \mathbb R[x]$ such that for any arbitrary $c$-closed graph $G$ on $n$ vertices,  and for each $(H_i, U_{i,+}, U_{i,-}) \in \mathcal{H}$, $$\# \{S\subset V(G) \colon  G[S] \text{ is a maximal }  (U_{i,+},U_{i,-})\text{-prescribed induced blow-up of $H_i$}\} \le f(n).$$

We say $\mathcal{H}$ is a \textbf{$c$-closed induced exponential pattern family}, if there is a constant $\epsilon>0$ and a sequence of $c$-closed graphs $(G_n)_{n\ge 1}$  such that $v(G_n) \to \infty$, and for each integer $n \geq 1$, there exists an $i$ such that  $$\# \{S\subset V(G_n) \colon  G_n[S] \text{ is a maximal }  (U_{i,+},U_{i,-})\text{-prescribed induced blow-up of $H_i$}\} \ge e^{v(G_n)^\epsilon}.$$  

\end{definition}

The next theorem is a corollary of Theorem \ref{thm:noninduced}.

\begin{theorem}\label{thm: finiteImplyPolynomial}
    If $\mathcal{H}$ is finite, then $\mathcal{H}$ is a $c$-closed polynomial pattern family. 
\end{theorem}

When counting induced patterns, i.e., $\mathcal{H} = \{ (H_1, U_{1,+}, U_{1, -}), \dots\}$ and when $\mathcal{H}$ is finite, we can also obtain a characterization theorem of when the family is polynomial versus when it is exponential. The exact statement is in Corollary \ref{cor: finiteFamilyInduced}.

When $\mathcal{H}$ is infinite, i.e., the patterns in $\mathcal{H}$ has size grow to infinity, then we have the following general results.

\begin{theorem}\label{thm:infinitebd}
    For any  $d > 1$, there exists $c > 0$ such that if each pattern $H$ in $\mathcal{H}$ has maximum degree bounded above by $d$, and $\mathcal{H}$ is an infinite family, then $\mathcal{H}$ is a $c$-closed exponential pattern family. 
\end{theorem}
To prove this Theorem, we will prove in 
 Section~\ref{sec: bounded} the following theorem. 
\begin{theorem}\label{thm: bounded}
    Let $d\ge 2$ be a positive integer. Let $H$ be a connected graph with maximum degree $d$ and $N$ vertices, where $N > d^{d^{d^d}}$. Then there exists a graph $G$ such that $N \leq |V(G)| \leq 2^{d^{21d+20}}N$, the maximum degree is at most $d2^{d^{21d+20}}$, and there are at least $2^{|V(G)|/2^{d^{21d+20}}}$ vertex sets in $G$ corresponding to maximal blow-ups of $H$.   
\end{theorem}

Note that the parameter $N$ is the pattern size, which grows arbitrarily large in infinite families. The host graph size is $\asymp_d$ $N$, so $2^N$ represents genuine exponential growth as pattern size increases. Also, the construction in the proof of theorem \ref{thm: bounded} applies to maximal $(H,U)$-blow-ups and maximal $(H,U_+,U_-)$-induced blow-ups. For the detailed statements, see theorems \ref{thm: boundedprescribed} and \ref{thm:boundedinduced}.

When patterns in $\mathcal{H}$ do not have bounded maximum degree,  the situation is more complicated. We will give examples of polynomial pattern families and exponential pattern families in section~\ref{sec:unbounded}.

\subsection{Difficulties}\label{subsec:diff}
The primary challenge arises from the difficulty of determining whether a given vertex set that forms a blow-up of a graph $H$ is maximal.
Unlike the case of enumerating maximal cliques, where verifying maximality is straightforward by checking that no other vertex is adjacent to all current vertices, the situation is more complex for blow-ups of a graph $H$ with at least two vertices. When a vertex is added to a set that induces a blow-up of $H$, the vertices may regroup in such a way that the resulting vertex partition still forms a larger blow-up. This potential for re-grouping is the central challenge in verifying maximality.

For example, consider the graphs $H$ and $G$ shown in Figure~\ref{fig:max-def}. One possible (but not maximal) blow-up of $H$ in $G$ maps $u_1$ to the group corresponding to $v_1$, $u_2$ to $v_2$, $u_3$ to $v_3$, and $u_4$ to $v_4$. In this configuration, $u_5$ cannot be added to any of the existing groups without violating the blow-up structure. However, $u_5$ can still be added to the set $\{u_1, u_2, u_3, u_4\}$ to form a larger blow-up of $H$, provided the {\it vertex partition is redefined}. In this larger blow-up, for example, $u_1$ corresponds to $v_1$, $u_2$ and $u_3$ are grouped together and correspond to $v_2$, $u_4$ corresponds to $v_3$, and $u_5$ corresponds to $v_4$.

\begin{figure}[H]
    \centering
    \includegraphics[scale=0.6]{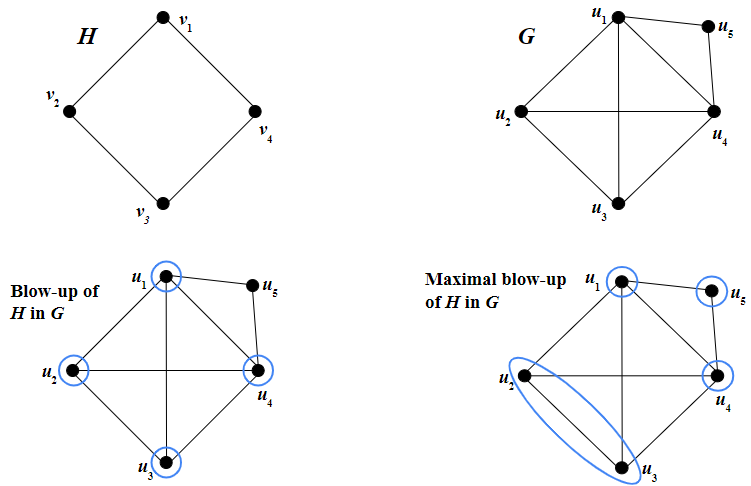}
    \caption{Example of extending a blow-up to maximality when searching for a blow-up of $H$ in $G$:
The blue circles represent the vertex groups in the blow-up that correspond to the vertices of 
$H$. The left diagram shows a proposed (non-maximal) blow-up, while the right diagram shows a maximal blow-up. Notably, reaching the maximal blow-up on the right requires regrouping the vertices from the smaller blow-up on the left; that is, the grouping in the maximal blow-up is not simply a superset of the original—it may involve reassigning vertices among groups.}
    \label{fig:max-def}
\end{figure}

This example shows that compared to finding maximal cliques, identifying maximal (prescribed) blow-ups of a general graph $H$ is more challenging. This added difficulty arises from the possibility of regrouping existing vertices and also potentially assigning newly added vertices to be its own group. 

Moreover, even if $V(G')$ is the vertex set of a maximal blow-up of a fixed graph $H$ in the host graph $G$, and assuming $G'$ contains a smaller induced subgraph $G''$ that is also a blow-up of $H$, it does not necessarily follow that adding any vertex from $V(G')\setminus V(G'')$ to $V(G'')$ will result in another blow-up of $H$. This behavior stands in stark contrast to cliques, where any subset of a clique is again a clique. Recall that cliques can be viewed as clique-prescribed blow-ups of a single vertex. Therefore maximality for cliques is much easier to determine and extensions are more straightforward.

An example illustrating this challenge is shown in Figure \ref{fig:hardmaximal}. The goal is to look for blow-ups of a 4-cycle $H$ within the given graph $G$. Let the vertices of $H$ be $v_1,v_2,v_3,$ and $v_4$, and let $u_1,u_2,u_3,u_4 \in V(G)$ be a blow-up of $H$ with $u_i$ being the only vertex in the group corresponding to $v_i$. Notice that $u_5$ and $u_6$ can each be easily added to join any of $u_1,u_2,u_3,$ or $u_4$ in the groups representing $v_1,v_2,v_3,$ and $v_4$ respectively. However, $u_7$ cannot be added to join an already existing group. In fact,  there is no maximal blow-up consisting of all the vertices of $G$ until $u_5$, $u_6$, and $u_7$ are each put in their own group and $u_1,u_2,u_3$, $u_4$ are grouped together to represent one vertex in $H$. This paper will address the case for general $H$ and overcome these difficulties. 
\begin{figure}[H]
    \centering
    \includegraphics[scale=0.6]{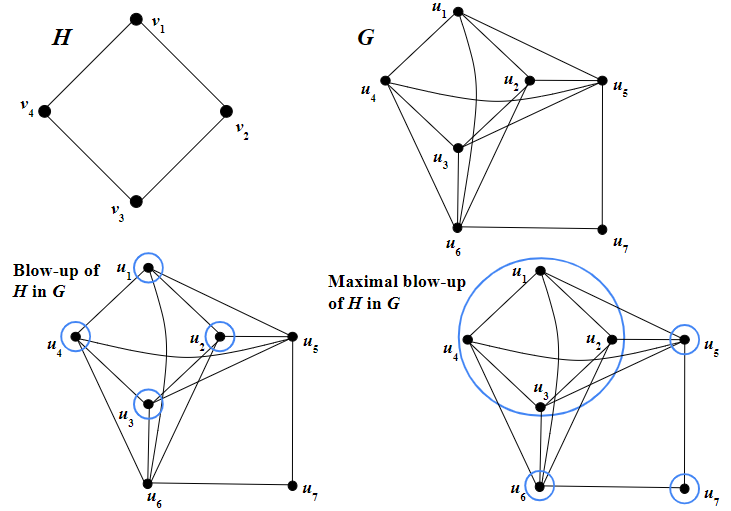}
    \caption{This example illustrates that for a general graph $H$, verifying maximal blow-ups of $H$ within a host graph $G$ can be challenging. The images below show two blow-ups: the left depicts a proposed (but not maximal) blow-up, while the right shows a maximal blow-up that includes all relevant vertices. Notably, a vertex set that is both contained in a blow-up and contains another blow-up is not necessarily itself a blow-up. For instance, $G[\{u_1, u_2, u_3, u_4, u_7\}]$ is not a blow-up of $H$, even though it is contained in the larger blow-up $G[\{u_1, \dots, u_7\}]$, and contains the smaller blow-up $G[\{u_1, \dots, u_4\}]$.}
    \label{fig:hardmaximal}
\end{figure}

\section{Enumerating maximal non-induced blow-ups}\label{sec: noninduced finite}

In this section, we show that in a $c$-closed graph on $n$ vertices, the maximal prescribed blow-ups of a fixed size graph is always at most polynomial in $n$. Below is the same statement as in Theorem \ref{thm:noninduced}. 
\begin{theorem}\label{thm: noninduced finite}
Let $H$ be a graph on $k \geq 2$ vertices and $U$ a subset of $V(H)$. 
For $c > 0$, the number of sets which form maximal $U$-prescribed blow-ups of $H$  in any $c$-closed graph on $n$ vertices is bounded above by $\left( n^{\max\{c-1,1\}}+\min \{3^{(c-1)/3}n^2, 4^{(c+4)(c-1)/2}n^{2-2^{1-c}}\}\right)^k$
\end{theorem}

\begin{proof}
Let $G$ be a $c$-closed graph on $n$ vertices. 
Let $S \subseteq V(G)$ be a maximal blow-up of a graph $H$ with $k$ vertices. In other words, there is a way to partition $S$ into disjoint nonempty vertex sets $A_1, A_2, ... , A_k$, where each $A_i$ corresponds to vertex $i$ in $H$. The partition is chosen so that the $L_2$ norm is maximized and we break the tie arbitrarily. We define the $L_2$ norm on the partitions to be $\sqrt{|A_1|^2 + |A_2|^2 + ... + |A_k|^2}$, so maximizing it means that it is not possible to make a partition where the $L_2$ norm on the sizes of the parts is larger than this one. 

First, we fix $T = \{ i \in [k] \colon |A_i| \ge c\}$. We write $j\sim i$ to indicate that vertices $i,j$ are adjacent. Observe for $i\in T$, note $G[\bigcup_{j\sim i} A_j]$ is a clique; for otherwise, there would be nonadjacent vertices with $A_i$ in their common neighborhood and by the definition of $c$-closedness, it would imply $|A_i| < c$. Then for any of these adjacent parts $A_j$ where $j \sim i$, all but one of the vertices in this part can be moved to $A_i$. Unless $|A_j|=1$ for all $j\sim i$, this would increase the $L_2$ norm while still keeping a blow-up of $H$ with prescribed parts being cliques, which leads to a contradiction. So if $|A_i| \geq c$ (i.e., $i \in T$), then all neighboring parts have size 1. Thus, $H[T]$ is an independent set, and now we fix $A_i$ for all $i\notin T$.

For each $i\in T$, let $S_i'=\cup_{j\sim i}A_j$ (note each $A_j$ is fixed as $|A_j|=1$, so this depends only on $i$, not on $A_i$).  Call $\tilde{N}(S_i')$ the common neighbors of $S_i'$ not including $S_i'$, i.e. $\tilde{N}(S_i')=(\cap_{s\in S_i'}N(s))\backslash S_i'$. If $i\not\in U$, we have $A_i\subset \tilde{N}(S_i')\subset A_1\cup A_2\cup\ldots \cup A_k$, for otherwise we can choose $v\in \tilde{N}(S_i')\setminus (A_1\cup A_2\cup\ldots \cup A_k)$ and add $v$ to $A_i$, contradicting maximality of blow-up. If $i\in U$, then $A_i$ has to be a clique. Choose a maximal clique $C_i$ such that $A_i \subseteq C_i$. Then $A_i\subseteq C_i\cap \tilde{N}(S_i')$, and $C_i\cap \tilde{N}(S_i')\subseteq A_1\cup A_2\cup \ldots \cup A_k$, or otherwise we can choose $v\in C_i\cap \tilde{N}(S_i')\setminus (A_1\cup A_2\cup \cdots \cup A_k)$ and add it to $A_i$.

Thus, 

$$  A_1\sqcup \cdots \sqcup A_k \subseteq \bigsqcup_{i\notin T} A_i \sqcup \bigsqcup_{i\in T\setminus U} \tilde{N}(S_i') \sqcup \bigsqcup_{i\in T\cap U}(C_i \cap \tilde{N}(S_i')) \subseteq  A_1\sqcup \cdots \sqcup A_k $$

So equality indeed holds. Thus, for every maximal blow-up $S$, there exists $T$, $(A_i)_{i\notin T}$, and $(C_i)_{i\in T\cap U}$ that uniquely determine $S$.

Given $|T|=m$, the number of choices of $(T, (A_i)_{i\notin T})$ is at most $\begin{pmatrix}k\\m\end{pmatrix} n^{\max\{c-1,1\}(k-m)}$ (note: this includes the possibilities when $|A_i| < c-1$ for some $i$). By Theorem 1, there are at most $\min \{3^{(c-1)/3}n^2, 4^{(c+4)(c-1)/2}n^{2-2^{1-c}}\}$ maximal cliques in $G$. Therefore, given $T,U$ with $|T|=m$, the number of choices for $(C_i)_{i\in T\cap U}$ is at most 
\[\min \{3^{(c-1)/3}n^2, 4^{(c+4)(c-1)/2}n^{2-2^{1-c}}\}^m.\] 
Thus, the number of maximal blow-ups is at most 
\begin{align*} \sum_{m=0}^k \binom km (n^{\max\{c-1,1\}})^{k-m}\min \{3^{(c-1)/3}n^2, 4^{(c+4)(c-1)/2}n^{2-2^{1-c}}\}^m  \\ = \left( n^{\max\{c-1,1\}} + \min \{3^{(c-1)/3}n^2, 4^{(c+4)(c-1)/2}n^{2-2^{1-c}}\}\right)^k.
\end{align*}
\end{proof}

We have the following corollary (Theorem \ref{thm: finiteImplyPolynomial}).
\begin{corollary}
    If $\mathcal{H} = \{(H_1, U_1), (H_2, U_2), \dots, (H_m,U_m)\}$ is finite, then $\mathcal{H}$ is a $c$-closed polynomial pattern family. 
\end{corollary}

\section{Enumerating maximal induced blow-ups}\label{sec: induced finite}
In this section, we prove the following theorem which characterizes which graphs $H$ satisfy that in any $c$-closed graphs on $n$ vertices, the number of maximal induced blow-ups of $H$ in $G$ is always a polynomial in $n$, with constants depending on $|V(H)|$ and $c$. Furthermore, we show there is a dichotomy. Fix a graph $H$, 

\begin{itemize}
    \item Either for any $c$, the number of maximal induced blow-ups of $H$ in any $c$-closed graph is bounded above by some polynomial in $|V(G)|$,
    \item Or there exists a $c$ and arbitrarily large $c$-closed graphs where the number of maximal induced blow-ups of $H$ in these graphs are at least exponential in $|V(G)|$. 
\end{itemize} 

We restate the characterization Theorem \ref{thm:inducedchar} below.  In $H$, recall a vertex $i$ is \textbf{bad} if its neighbors form a clique in $H$.

\begin{theorem}
    Let $H$ be a graph on $k$ vertices where $k \geq 1$. Let $U_+, U_- \subseteq V(H)$ be two disjoint subsets. 
    \begin{enumerate}
        \item Suppose in $H$ there are no bad vertices, or exactly one bad twin group $B$ (note $H[B]$ is either a clique or an independent set) and at least one of the following hold:
        
        \begin{enumerate}
            \item If $H[B]$ is a clique, and there is at least one unprescribed vertex in $B$,
            \item If $H[B]$ is a clique, and $B\subseteq U_+$. This includes the case where $|B|=1$ and the unique vertex in $B$ is in $U_+$. 
            \item If $H[B]$ is an independent set, there is exactly one unprescribed vertex in $B$, and every other vertex in $B$ is in $U_-$.
            
        \end{enumerate}
        
        Then, for any $c>0$ and in any $c$-closed graph $G$ on $n$ vertices, the number of vertex sets corresponding to maximal induced $(U_+, U_-)$--prescribed blow-ups of $H$ in $G$ is at most $(2n)^k (n^22^c)^{k}$.
        \item Suppose in $H$ either there are at least two bad twin groups; or there is exactly one bad twin group $B$ and one of the following hold:
        
        \begin{enumerate}
            \item $H[B]$ is a clique, $B\subseteq U_- \sqcup U_+$ and $B\cap U_- \ne \emptyset$.
            \item $H[B]$ is an independent set and $B\subseteq U_- \sqcup U_+$ and $B\cap U_- \ne \emptyset$.

            \item $H[B]$ is an independent set, and there exists at least two vertices in $B$ that are not in $U_-$.
           
        \end{enumerate}
        
        Let $k:=v(H)$. For any $n\ge 1000k^5$, there is a $(k+1)$-closed graph $G$ with $v(G) \in \{n,n+1\}$ such that the number of vertex sets corresponding to maximal $(U_+, U_-)$--prescribed induced blow-ups of $H$ is at least $2^{-3k^2-2k+\frac n2}$. If furthermore $H$ is $c$-closed, then $G$ can be constructed to be $(c+2)$-closed.
    \end{enumerate}
\end{theorem}

\
 We first prove the first statement in Theorem \ref{thm:inducedchar}. 
\begin{proof}[Proof of Theorem \ref{thm:inducedchar} Statement 1.]
    We first handle the case where $k=1$, i.e., $H$ is a single vertex. If the single vertex is in $U_-$, then the graph does not satisfy the hypotheses of statement 1. If the single vertex is in $U_+$, then by \cite{paper1}, the number of blow-ups is at most $3^{c/3}n^2 \le 2^cn^3$. If the single vertex is not prescribed, then the only maximal induced blow-up of a single vertex is the whole set, and thus the result trivially holds. 

    From now on we assume $k = |V(H)| \geq 2$, and in this proof, for the sake of conciseness, when we say an $H$-blow-up, we mean an induced $H$-blow-up that is $(U_+,U_-)$-prescribed. 
    For a given $A\subset V(G)$ such that $G[A]$ is a maximal $H$-blow-up, we say a partition $A_1 \sqcup A_2 \sqcup \dots \sqcup A_k$ of $A$ is \textbf{excellent} if
    for any other partition $A'_1 \sqcup A'_2 \sqcup \dots \sqcup A'_k$ that makes it an $H$-blow-up, either $\sum |A'_i|^2 < \sum|A_i|^2$, or $\sum |A'_i|^2 = \sum |A_i|^2$ and there exists $j$ such that $|A_{j'}| = |A'_{j'}|$ for all $j'=1,\cdots,j-1$ and $|A_j|>|A'_j|$. 
   
   For each $i \in [k]$, arbitrarily pick one vertex $v_i$ from each set $A_i$. There are at most $n^k$ choices for these $(v_i)_{i=1}^k$. Furthermore, we let $(c_i)_{i=1}^k$ be variables where $c_i = 1$ forces $A_i$ to be a clique, and $c_i = 0$ forces $A_i$ to be not a clique.

    We say an index $i\in [k]$ is \textbf{challenging with respect to a fixed} $((v_j,c_j))_{j=1}^k$ where $v_j \in V(G)$ and $ c_j \in \{0,1\}$ if the following hold: The number of choices of $A_i$ in an excellent partition, such that 
    \begin{enumerate}\item $v_l \in A_l$ for all $l \in [k]$, and \item $G[A_l]$ is a clique if and only if $c_l = 1$ for all $l\in [k]$,\end{enumerate} is larger than $n^2 2^c$, i.e.

    \[
  \# \left\{ A_i\ \middle\vert \begin{array}{l}
    \exists (A_l)_{l\in [k]\setminus i}\textup{ such that } v_l\in A_l \ \forall \ l \in [k], \\
    \textup{and }(A_l)_{l\in [k]} \textup{ is an excellent partition},  \\
    \textup{and } G[A_l] \textup{ clique } \iff c_l=1 \forall \ l\in [k]\\
    
  \end{array}\right \}> n^22^c
\]

In particular, $|A_l|=1$ forces $c_l = 1$, and we still require $G[A_i]$ to be a clique or not a clique depending on whether $c_i$ is $1$ or $0$. 
Note that the definition of excellent partition is independent from $c_i$ and $v_i$. Hence when we test whether $(A_l)_{l\in [k]}$ is excellent by comparing it with any other partition $A'_1 \sqcup A'_2 \sqcup \dots \sqcup A'_k$ that makes it an $H$-blow-up, we do not require $v_i\in A_i'$ and allow $G[A'_i]$ to be a clique or non-clique regardless of the value of $c_i$.

\begin{claim}\label{claim:1prime}
        Let $G$ be an arbitrary $c$-closed graph with $n$ vertices. Fix $((v_l,c_l))_{l=1}^k$.  Then 
         if $i$ is challenging with respect to $((v_l,c_l))_{l=1}^k$, then $i$ is bad in $H$.

    \end{claim}
    \begin{proof}
      Fix any vertex $i$ in $H$. Suppose $i$ is not a bad vertex, i.e.,  the neighborhood of $i \in V(H)$ contains a non-edge, say vertex pair $(j, s)$. Then $A_j$ is empty to $A_s$ in $G$, while $A_i$ is in the common neighborhood of vertices $v_j \in A_j$ and $v_s \in A_s$. Therefore by the definition of $c$-closedness, $|N(v_j) \cap N(v_s)| \le c$. Thus the number of choices for $A_i$ is at most $2^c$ when fixing $v_j, v_s$, and thus $i$ cannot be challenging. 
      \end{proof}

    We return to the proof of Theorem \ref{thm:inducedchar}.  It suffices to show that for any $((v_l,c_l))_{l=1}^k$, there exists at most $(n^22^c)^k$ excellent partitions $(A_1,\cdots ,A_k)$ such that $G[A_1\sqcup \cdots \sqcup A_k]$ is a maximal $H$-blow-up, $v_i\in A_i $ and $G[A_i]$ is a clique if and only if $c_i=1$. Having fixed $((v_l,c_l))_{l=1}^k$, let $C \subseteq [k]$ be the set of $i$ where $i$ is challenging. 
    
    Fix $((v_l,c_l))_{l=1}^k$.
If $|C| = 0$, then the number of choices of $(A_1,A_2,\cdots,A_k)$ that form maximal prescribed induced blow-ups of $H$ is at most $(n^22^c)^k$, since there are at most $n^22^c$ possibilities for each $A_i$. Henceforth assume $|C|\ge 1$.
    
    By claim \ref{claim:1prime}, for all $i\in C$, $i$ is bad. Thus these $i$'s must come from the bad twin group. 

    \noindent\textbf{Case 1: $H[B]$ is a clique.} This corresponds to scenarios a), b) of part 1 of Theorem \ref{thm:inducedchar}.

    We now show another claim.

    \begin{claim}
        Under the hypothesis of Case 1, let $G$ be an arbitrary $c$-closed graph on $n$ vertices. Fix $((v_l,c_l))_{l=1}^k$. Suppose $i<j\in V(H)$ are both challenging with respect to $((v_l,c_l))_{l=1}^k$. Then they are not adjacent in $H$.
       
    \end{claim}
    \begin{proof}
         
    Assume for the sake of contradiction $i,j$ are adjacent. If $c_j = 0$, then there must exist $v_j', v_j'' \in A_j$ such that $v_j', v_j''$ are not adjacent. Since $i,j$ are adjacent, for each fixed $v_j',v_j''$ that are not adjacent there are at most $2^c$ choices for $A_i$. Thus, by considering all possible choices of $(v_j', v_j'') \notin E(G)$, there are at most $n^22^c$ choices for $A_i$, contradicting $i$ is challenging. Analogously, if $c_i=0$, $j$ is not challenging, so from now on we only need to deal with the case when $c_i=c_j=1$.
    
    By claim 1, $i,j$ are necessarily bad, and recall there is only one bad plus twins group, as $H[B]$ is a clique.

     First assume $i \notin U_-$. Thus if $|A_j|>1$, then we can move all but one vertex in $A_j$ to $A_i$, 
    still resulting in a prescribed induced blow-up of $H$, while increasing $|A_1|^2+ \dots + |A_k|^2$ (or this quantity stays the same, but $(|A_1|,\cdots,|A_k|)$ is moved up in the lexicographic order, for $i<j$). Thus, if $A_1\sqcup \cdots \sqcup A_k$ is an excellent partition, then $|A_j|=1$, so this implies $j$ is not challenging with respect to $((v_l,c_l))_{l=1}^k$.
     
     It remains to handle the case when $i \in U_-$. Since $c_i=1$, $G[A_i]$ is both a clique and an independent set, so $|A_i|=1$, so there are at most $n$ choices for $A_i$. This implies $i$ is not challenging. 
    \end{proof}

    Thus, in the Case 1 where $H[B]$ is a clique, $|C|\le 1$ by the hypothesis of statement 1 that there is at most one bad plus twin group. We already handled $|C|=0$, so only need to deal with the case when $|C|=1$. From now on we ignore lexicographic conditions/constraints on excellent partitions, so without loss of generality, assume index $1$ is challenging.

Recall it suffices to only handle the case when  $|C|=1$. First note that the number of choices for $(A_s)_{2 \leq s \leq k}$ is at most $(n^22^c)^{k-1}$. Thus, it suffices to show that for every choice of $(A_s)_{2\le s\le k}$ 
and every choice of $(v_1,c_1)$, there is at most $n^22^c$ choices for $A_1$. 

If $1$ is un-prescribed, then the only choice for $A_1$ that makes $G[A_1\sqcup \cdots \sqcup A_k]$ a maximal $H$-blow-up is the set of vertices in $G$ which have the correct adjacency or non-adjacency to each of the vertices in $A_s$ for $2\leq s \leq k$. Thus the choice of $A_1$ is actually unique (if well-defined with respect to $(v_1,c_1)$). 

If $1 \in U_+$, i.e., $A_1$ needs to be a clique, then the number of choices of $A_1$ is at most the number of maximal cliques in the set of vertices in $G$ which have the correct adjacency or non-adjacency to each of the vertices in $A_s$ for $2\leq s \leq k$. Thus the number of choices for $A_1$ is at most $\min \{3^{(c-1)/3}n^2, 4^{(c+4)(c-1)/2}n^{2-2^{1-c}}\} \leq 2^c n^2$ by \cite{paper1}. This handles case (b), where $B\subseteq U_+$. 
    
    The last case is $1 \in U_-$. This tells us we are in case (a), since in case (b), $B\subseteq U_+$. In case (a), vertex $1$ has at least one more bad plus twin, and that vertex is not in $U_+ \cup U_-$. Say this bad plus twin group is $\{1,\cdots,s\}$ and vertex $2$ is unprescribed. Now we can also forget about the $c_i$'s.
    \begin{itemize}
        \item If there exists $l \in \{2,\cdots,s\}$ such that $A_l$ contains a non-edge, then $A_1$ is contained in the common neighborhood of an non-edge in $A_l$, and thus the number of choices of $A_1$ is at most $2^c$. 
        \item Otherwise, $G[A_2],\cdots,G[A_s]$ are all cliques (*). We claim that after fixing $A_2, \dots, A_{s}$, there is at most one option for $A_1 \cup \dots \cup A_{s}$. This means that there is at most one option for $A=A_1\cup \cdots \cup A_k$.
        
        Since $A_{s+1}, \dots, A_k$ are already fixed, $A_1$ must satisfy that $A_1 \cup \dots \cup A_s$ is a partition which satisfies that $G[A_i,A_j]$ is complete bipartite for all $1\le i<j\le s$, and $G[A_i]$ is a clique (or an independent set) if $i\in U_+,U_-$ respectively. 
        
        By (*), $G[A_2 \cup \dots \cup A_s]$ is a clique. Recall that $2\notin U_+ \cup U_-$. Let $S$ be the set of vertices in $G \setminus \bigcup_{j=2}^k A_j$ which are in the common neighborhood of the $\cup_{j\sim 1} A_j$ and not adjacent to any vertex in $\bigcup_{j \nsim 1} A_j$.  Clearly $A_1\subseteq S$, so $A_1\cup \cdots \cup A_s \subseteq S \cup A_2 \cup \cdots \cup A_s$. On the other hand, $G[S \cup A_2 \cup \cdots \cup A_k]$ can be made an $H$-blow-up as follows: first consider the partition $A'_1\sqcup \cdots \sqcup A'_k$ where $A'_1 = A_2, A'_2 = S$ and $A'_j = A_j$ for all $3\le j\le k$. Since $\{1,\cdots,s\}$ is a bad plus twin group, the only reason why this partition is not an $H$-blow-up is because $G[A'_1]=G[A_2]$ might not be an independent set (in fact, $G[A_2]$ is a clique), while $1\in U_-$. To fix this, note that since $1,2$ are plus twins in $H$, we move all vertices but one from $A'_1$ to $A'_2$, then $G[A'_1]$ is a singleton, and $G[A'_1,A'_2]$ remains complete because $G[A'_1]$ is a clique before moving. Since $2 \notin U_-\cup U_+$, we obtain a blow-up of $H$. 

    Thus, if $A_2,\cdots,A_k$ are fixed and $G[A_i]$ is a clique for all $2\le i\le s$, then for any $A_1$, $A_1 \cup A_2\cup \cdots \cup A_k \subset S\cup A_2\cup \cdots \cup A_k$. Note that $S\cup A_2\cup \cdots \cup A_k$ is a $H$-blow-up, so if $G[A_1\sqcup \cdots \sqcup A_k]$ is a maximal $H$-blow-up, then $A_1$ must be equal to $S$.
    \end{itemize} 

    \textbf{Case 2: $H[B]$ is an independent set. This corresponds to statement 1(c).}
    
    In this case, this bad twin group contains exactly one unprescribed vertex, and every other vertex in $B$ is in $U_-$.

    We now ignore lexicographic constraints on excellent partitions. Label the vertices of the bad twin group as $\{i_1,\cdots,i_t\}$, where $i_1$ is an unprescribed vertex and $i_2,\cdots,i_t $ are in $U_-$. For a partition $A_1\sqcup \cdots \sqcup A_k$ that is an
    $H$-blow-up, if for all $j=2,\cdots,t$, we move all but one vertex from $A_{i_j}$ to $A_{i_1}$, we still get an $H$-blow-up. 
    
    Now we count the number of choices for $(A_s)_{s\ne i_1}$ such that $|A_{i_j}| = 1$ for all $2\le j\le t$ and there exists $A_{i_1}$ such that $\bigsqcup_{l=1}^k A_l$ forms a maximal $H$-blow-up. 
    
    Note there are at most $(n^22^c)^{k-1}$ choices for $(A_s)_{s\ne i_1}$. For each choice of $(A_s)_{s\ne i_1}$, we have to choose $A_{i_1}$ to be the set of vertices that have the correct adjacency and nonadjacency relationships to vertices in $(A_s)_{s\ne i_1}$, since $i_1 \notin U_-\sqcup U_+$. 
    
   Thus, there are at most $(n^22^c)^{k-1}$ choices for $k$-tuples $(A_1,\cdots,A_k)$ such that $A_1\sqcup \cdots \sqcup A_k$ is a maximal $H$-blow-up. Since $(A_1,\cdots,A_k) \mapsto \bigsqcup_{l=1}^k A_l$ is a surjection from these $k$-tuples of vertex sets to maximal $H$-blow-ups, the number of maximal $H$-blow-ups is at most $(n^22^c)^k$, as desired.
    
\end{proof}

\

\begin{proof}[Proof of Theorem \ref{thm:inducedchar} Statement 2.]

Recall $k:=v(H)$. For the other cases where we show there exists $k$-closed graphs $G$ on $n$ vertices such that there exists at least $2^{-3k^2-2k+\frac n2}$ maximal induced $(H,U_+,U_-)$-blow-ups, the key strategy in all the constructions is to somehow stick a matching such that if we choose exactly one vertex from each matching, we get a blow-up, and in any blow-up, there are a bounded number of pairs (for example $\le 3k^2$) in the matching. Furthermore, if $H$ is $c$-closed, our construction guarantees that $G$ is $(c+2)$-closed.

\

\begin{proposition}\label{prop:unifying}
    Suppose the graph $G$ satisfies the following condition: there exists vertices $a_1,a_2,\cdots,a_K,b_1,\cdots,b_K \in V(G)$ and a set of vertices $U\subseteq V(G), U \cap \{a_1,\cdots,a_K,b_1,\cdots,b_K\} = \emptyset$ such that 

    \begin{itemize}
        \item If $V$ is a subset of $ \{a_1,a_2,\cdots,a_K,b_1,\cdots,b_K \}$ such that $V \ne \{a_1,\cdots,a_K\}, V\ne \{b_1,\cdots,b_K\}$ and for each $i\in \{1,\cdots,K\}$, exactly one of $a_i,b_i$ is in $V$. Then $V\sqcup U$ is a (not necessarily maximal) induced $(H,U_+,U_-)$-blow-up.
        \item Any induced $(H,U_+,U_-)$-blow-up containing $U$ and at least one of $a_i,b_i$ for every $i$ must contain at most $K + 3k^2$ elements in $\{a_1,\cdots,a_K,b_1,\cdots,b_K\}$. 
    \end{itemize}
    Then there are at least $2^{-3k^2}(2^K-2)$ vertex sets that form maximal $(H,U_-,U_+)$-blow-ups in $G$.
\end{proposition}

\begin{proof}
    for every nonempty $S\subsetneq \{1,\cdots,k\}$, we can find $f(S)\subseteq V(G)$ such that $A_S \subseteq f(S)$, $f(S)$ is a maximal $H$-blow-up, and by the claim, $|f(S) \setminus A_S| \le 2k$. Consider the mapping $$\mathcal{P}(\{1,\cdots,k\}) \setminus \{ \emptyset, \{1,\cdots,k\}\} \to \{ \text{maximal }H\text{-blow-ups}\}, S \mapsto f(S)$$ 

   Note that the fibers of $f$ have size at most $2^{2k}$. This is because a maximal $H$-blow-up $X$ in the image of $f$ must contain $V(H) \setminus \{v_1,v_2\}$ and at least one of $a_i,b_i$ for every $1\le i\le K$, and there must exist at most $2k$ indices $i$ such that both $a_i,b_i$ are contained in $X$. Now suppose $X=f(S)$, then $A_S$ must be obtained by removing exactly one of $a_i,b_i$ for every $i$ such that $a_i,b_i\in X$. Thus, $|f^{-1}(\{X\})| \le 2^{2k}$ for all $X$. Since the source of $f$ has $2^K-2$ elements,  $$|Im(f)| \ge \left\lceil\frac{2^K-2}{2^{2k}}\right\rceil = 2^{K-2k} > 2^{\frac n2 - 2k - 3k^2}.$$
\end{proof}
We first handle the case where $H$ has $b \geq 2$ bad twin groups.

 \begin{claim}\label{claim:disjoint}
     If $v,w$ belong to two different bad twin groups, then $v\not\sim w$. 
 \end{claim}
 \begin{proof}
     We prove by contradiction. If $v\sim w$, then $v \in N(w)$. Since $w$ is bad, $H[N(w)]$ is a clique, so $N(v) \cup \{v\} \supseteq N(w) \cup\{w\}$. Analogously, $N(w) \cup\{w\}\supseteq N(v) \cup \{v\}$, so $v,w$ belong to the same bad plus twin group, which leads to a contradiction.
 \end{proof}

We now divide this into three subcases.

\begin{enumerate}
    \item There are at least two different bad twin groups $B_1\not\subseteq U_-$ and $B_2\not\subseteq U_-$. 
    \item There exists exactly one bad twin group $B$ such that $B\not\subseteq U_-$. 
    \item All bad vertices are in $U_-$. 
\end{enumerate}

\newpage 

    \noindent\textbf{Case $1$: $H$ has at least two bad twin groups, and there are at least two different bad twin groups that are not contained in $U_-$.}

 \begin{figure}
     \centering
     \includegraphics[scale=0.6]{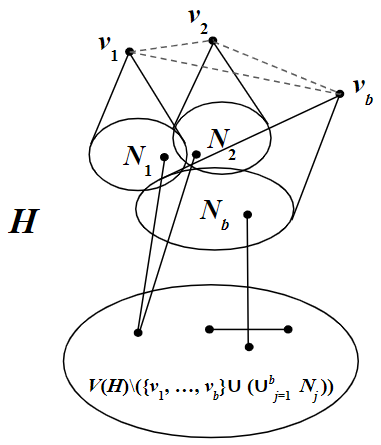}
     \caption{Graph $H$.  The vertices $v_1, \dots, v_b$ are bad and form an independent set. Each $N_1, \dots, N_b$ is a clique and is disjoint from $\{v_1, \dots, v_b\}$. }
     \label{fig:inducedH}
 \end{figure}

  \begin{figure}
     \centering
     \includegraphics[scale=0.6]{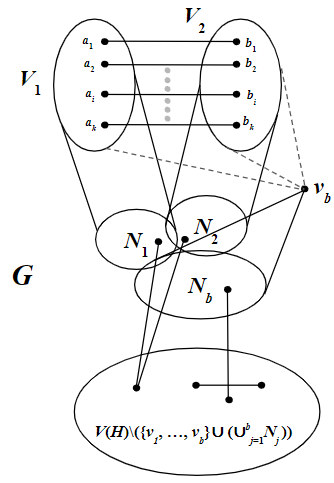}
     \caption{Construction of a $k+1$-closed graph $G$ given $H$. For $i = 1,2$, the vertices $v_i$ is replaced by a clique $V_i$ of size $K$ if $i \in U_+$ or $i \notin U_-\cup U_+$, and respectively and the non-edge between $v_1, v_2$ in $H$ is replaced by a perfect matching $\{(a_1, b_1), \dots, (a_K, b_K)\}.$}
     \label{fig:inducedG}
 \end{figure}

Let $B_1,B_2$ be two distinct bad twin groups not contained in $U_-$. We choose $v_1 \in B_1 \setminus U_-$ and $v_2\in B_2\setminus U_-$ arbitrarily. Define $K := \lceil \frac{n-(k-2)}{2}\rceil$.
 
\textbf{Construction:} For each $n$, we now construct a $(k+1)$-closed graph on $n$ or $n+1$ vertices as follows. 
    Starting from the original graph $H$ (Figure \ref{fig:inducedH}), replace $v_1$ by a clique $\{a_1,\cdots,a_K\}$, and replace $v_2$ by a clique $\{b_1,\cdots,b_K\}$ of size $K$. In between there two cliques, there is a perfect matching $\{(a_1, b_1), \dots, (a_K, b_K)\}$, and add edges from each vertex in $N(v_1)$ to each vertex in $\{a_1,\cdots,a_K\}$ and add edges from each vertex in $N(v_2)$ to each vertex in $\{b_1,\cdots,b_K\}$. 
    See Figure \ref{fig:inducedG} for an illustration. In this construction, we treat $H \setminus \{v_1, v_2\}$ as an induced subgraph of $G$ on $V(G) \setminus \{a_1,\cdots,a_K,b_1,\cdots,b_K\}$, where by a slight abuse of notation, each vertex $v \in V(H) \setminus \{v_1, v_2\}$ is also vertex $v$ in $G$. 

    Observe that if $H$ is $c$-closed, then we claim $G$ is $(c+1)$-closed. Recall $G$ is $\alpha$-closed if $\max\limits_{u \not\sim v, u\ne v} |N(u) \cap N(v)| < \alpha$. 

    \begin{itemize}
        \item $u,v \notin \{a_1,\cdots,a_K,b_1,\cdots,b_K\}$: in this case, if $a_i \in N(u)\cap N(v)$, then in $H$, $u\sim v_1$ and $v\sim v_1$. Since $v_1$ is bad, this implies $u\sim v$. Thus, if $u \not\sim v$, then none of the $a_i$'s or the $b_j$'s are in $N(u)\cap N(v)$, so $|N_G(u) \cap N_G(v)| = |N_H(u)\cap N_H(v)| < c$.
        \item $u \in \{a_1,\cdots,a_K,b_1,\cdots,b_K\}, v\notin \{a_1,\cdots,a_K,b_1,\cdots,b_K\}$. WLOG $u \in \{a_1,\cdots,a_K\}$, for the case where $u \in \{b_1,\cdots,b_K\}$ is handled in a symmetric way. If $u\not\sim v$, then in $H$, $v_1 \not \sim v$. Thus, $v$ is not adjacent to any of $a_1,\cdots,a_K$ in $G$. Thus, if $u=a_i$, then $N_G(u)\cap N_G(v) \subseteq (N_H(v_1) \cap N_H(v)) \cup \{b_i\}$, which has cardinality less than $c+1$. 
        \item $u,v \in \{a_1,\cdots,a_K,b_1,\cdots,b_K\}$. In this case, if $u\not \sim v$, then there exists $i\ne j$ in $[K]$ such that $u = a_i, v=b_j$. Then $N_G(a_i)\cap N_G(b_j) \subseteq \{a_j,b_i\} \cup (N_H(v_1)\cap N_H(v_2)) $, which has at most $c+2$ elements. 
    \end{itemize}
    
   We now point out $2^K-2$ specific (not necessarily maximal) induced blow-ups of $H$ in $G$. 
    For each nonempty subset $S \subsetneq \{1,\cdots,K\}$  we define $$ A_S := \{a_i\colon i\in S\} \bigsqcup \{ b_i \colon i\notin S\} \bigsqcup (V(H)\setminus \{v_1,v_2\})$$ 
    
    Note $G[A_S]$ is a prescribed induced blow-up of $H$. Indeed, it is the case by setting $A_{v_1} = \{a_i\colon i\in S\}, A_{v_2} = \{ b_i \colon i\notin S\}$ and $A_x = \{x\}$ for all $x\in V(H)\setminus \{v_1,v_2\}$.

    Clearly $|A_S| = |V(G)| - K$ and $A_S \ne A_{S'}$ if $S\ne S'$.

    It is possible that each $G[A_S]$ is already a maximal induced blow-up of $H$. But we do not have a direct proof of it. Instead, we will prove the following. 
   \begin{claim}\label{claim:inducedmain}
       If $A\subseteq G$ is a  $H$-blow-up, then there exists at most $2k$ indices $i \in [K]$ can satisfy that both the vertices $a_i, b_i$ are in $A$. 
   \end{claim}

   Assuming Claim \ref{claim:inducedmain}, we can deduce the theorem by applying proposition \ref{prop:unifying} with $U = V(H)\setminus \{v_1,v_2\}$.

   The rest of the proof is dedicated to prove Claim \ref{claim:inducedmain}. We prove by contradiction. 
Suppose there are at least $2k$ indices $j$ such that the vertices $a_j$ and $b_j$ are both in $A$. By the pigeonhole principle, there exists $z\in V(H)$ and at least $2k / |V(H)| = 2$ indices $j$ such that $a_j\in A_z$. Call the indices $j\ne j'$.  Note that $b_j \sim a_j$ but $b_j \not\sim a_{j'}$. Thus, $b_j\in A_z$. Analogously, $b_{j'} \in A_z$. Now, for any $i\notin \{j,j'\}$ such that $a_i \in A$, note that $a_i \sim a_j$ but $a_i\not\sim b_j$. Since $a_j,b_j\in A_z$, it follows that $a_i\in A_z$. Analogously, if $b_i\in A$, then $b_i\in A_z$.

Now note $\{a_1, \dots, a_K, b_1, \dots, b_K\} \cap A \subseteq A_z$, we still need to assign at least one vertex to every $A_y$ for $y\in V(H) \setminus \{z\}$. However, there are only at most $v(H)-2$ vertices in $A \setminus \{a_1, \dots, a_K, b_1, \dots, b_K\} = V(H) \setminus \{v_1,v_2\}$. We can only use $v(H)-2$ vertices, and need to place at least one vertex in $v(H)-1$ parts of the blow-up. This is a contradiction. 

\

    \noindent\textbf{Case $2$: $H$ has at least two bad twin groups, and there is exactly one bad twin group that is not completely contained in $U_-$.}

Let $W$ be a bad twin group contained in $U_-$, and $B$ be the bad twin group not contained in $U_-$. We choose $v \in B\setminus U_-$, and choose $w\in W$ arbitrarily. Note $v\not\sim_H w$. 

\textbf{Construction:} Recall $k:=v(H)$. Let $K  := \lceil \frac{n-(k-2)}{2} \rceil > 20k^5$. We construct a graph $G$ on $k-2+2K $ vertices as follows: we start with $H$, replace $v$ with a clique $\{a_1,\cdots,a_K\}$, connect every $a_i$ with only the neighbors of $v$. We replace $w$ with an independent set $\{b_1,\cdots,b_K\}$. We connect $b_1,\cdots,b_K$ with the neighbors of $w$, and we connect $a_i$ with $b_i$ for all $1\le i\le K$. Call the resulting graph $G$. Note that $G$ is $(k+1)$-closed, and if $H$ is $c$-closed then $G$ is $(c+2)$-closed. 

Note that for any nonempty subset $S \subsetneq \{1,\cdots,K\}$, setting $A_v = \{ a_i \colon i\in S\}, A_w  = \{ b_i \colon i\notin S\}$ and $A_x = \{x\}$ for all $x\in V(H) \setminus \{v,w\}$ does make $\bigsqcup_{u\in V(H)} A_u$ an induced $H$-blow-up. 

\begin{claim}
    If $A$ is a
    $H$-blow-up, then there exists at most $2k$ indices $i$ such that $a_i,b_i\in A$.
\end{claim} 
\begin{proof}
    
By the pigeonhole principle, there exists $l\in V(H)$ and at least $2k / |V(H)| = 2$ indices $j$ such that $a_j\in A_l$. Call the indices $j,j'$.  Note that $b_j \sim a_j$ but $b_j \not\sim a_{j'}$. Thus, $b_j\in A_l$. Analogously, $b_{j'} \in A_l$. Now, for any $i\notin \{j,j'\}$ such that $a_i \in A$, note that $a_i \sim a_j$ but $a_i\not\sim b_j$. Since $a_j,b_j\in A_l$, it follows that $a_i\in A_l$. Thus $\{a_1,\ldots,a_K\}\cap A\subset A_l.$  Hence $|A_l|\geq |\{a_1,\ldots,a_K\}\cap A| \ge 2k$ so $l$ is bad. Since $a_j,a_{j'},b_j \in A_l$ and $a_j\sim b_j$ but $b_j\not\sim a_{j'}$, $A_l$ is neither a clique nor independent set, so $l$ is not prescribed. By assumption of Case 2, $l\in B$ thus $l$ is a twin of $v$, and $|N_H(l)| = |N_H(v)|$. 

 Note that the common neighborhood of $a_i,b_i$ in $G$ corresponds precisely to $N_H(v) \cap N_H(w)$, so at most $|N_H(v) \cap N_H(w)|$ vertices can be assigned to $\bigsqcup_{x\in N_H(l)} A_x $. Since each such $A_x$ has to be assigned at least one vertex, it follows that $|\{ A_x \colon x\in N_H(l)\}|=|N_H(l)|\leq |N_H(v) \cap N_H(w)|$. Since $|N_H(l)| = |N_H(v)|$, we have $|N_H(v)|\leq |N_H(v) \cap N_H(w)|$. This implies $ N_H(v)\subseteq N_H(w)$. 

However, since $v,w$ belong to different bad twin groups and $v\not\sim w$, we have $N_H(v) \subsetneq N_H(w)$. Choose $z \in N_H(w) \setminus N_H(v)$. Note that $z\sim b_i$ but $z\not\sim a_i$ for some pair $a_i,b_i\in A$, and we have shown before that $a_i,b_i\in A_l$. It follows that $z\in A_l$.

Now note for any $j$ such that $b_j \in A$, note $b_j$ is adjacent to $z\in A_l$, but also not adjacent to some other $b_i\in A_l$. Thus, $\{b_1,\cdots,b_K\} \cap A \subseteq A_l$. 

From $\{a_1,\cdots,a_K, b_1,\cdots,b_K\} \cap A \subseteq A_l$, we now have $\le v(H)-2$ other vertices in $A$ to be placed in $v(H)-1$ parts, contradiction.
\end{proof}

By applying proposition \ref{prop:unifying} on $U = V(H) \setminus \{v,w\}$, $G$ has at least $2^{-3k^2} (2^K-2)$ maximal $H$-blow-ups. 

\

    \noindent\textbf{Case $3$: $H$ has at least two bad twin groups, both contained in $U_-$.}

Let $B,W$ be arbitrarily chosen bad twin groups, and choose $v\in B, w\in W$ arbitrarily. Note $v \not\sim w$. 

\textbf{Construction:} Recall $k:=v(H)$. Let $K  := \lceil \frac{n-(k-2)}{2} \rceil > 20k^5$. We will construct a graph $G$ with $k-2+2K$ vertices as follows. We start with $H$, replace $v$ with a an independent set $\{a_1,\cdots,a_K\}$, connect every $a_i$ with only the neighbors of $v$. We replace $w$ with an independent set $\{b_1,\cdots,b_K\}$. We connect $b_1,\cdots,b_K$ with the neighbors of $w$, and we connect $a_i$ with $b_i$ for all $1\le i\le K$. Call the resulting graph $G$. Note that $G$ is $(k+1)$-closed. 

Note that for any nonempty subset $S \subsetneq \{1,\cdots,K\}$, setting $A_v = \{ a_i \colon i\in S\}, A_w  = \{ b_i \colon i\notin S\}$ and $A_x = \{x\}$ for all $x\in V(H) \setminus \{v,w\}$ does make $\bigsqcup_{u\in V(H)} A_u$ an induced $H$-blow-up. 

\begin{claim}
    If $A$ is a $H$-blow-up, then there exists at most $k^2$ indices $i$ such that $a_i,b_i\in A$.
\end{claim} 
\begin{proof}

Suppose more than $k^2$ indices $i$ such that $a_i,b_i\in A$. By the pigeonhole principle, there exists $l\in V(H)$ and at least $\lceil (k^2+1) / |V(H)| \rceil = k+1$ indices $j$ such that $a_j\in A_l$. Call the indices $j_1,j_2,\ldots,j_{k+1}$.  Note that $b_{j_2} \sim a_{j_2}$ but $b_{j_2} \not\sim a_{j_1}$. Thus, $b_{j_2}\in A_l$, and analogously $b_{j_t} \in A_l$ for all $t\in [k+1]$. Hence $|A_l|\geq 2k$ and $l$ is bad by the construction. But $A_l$ is not an independent set, contradicting the assumption that each bad vertex belongs to $U_-$.
\end{proof}
By applying proposition \ref{prop:unifying} on $U = V(H) \setminus \{v,w\}$, $G$ has at least $2^{-3k^2} (2^K-2)$ maximal $H$-blow-ups. 

\

    \noindent\textbf{Case 4: $H$ has exactly one bad twin group $B$, and $H[B]$ is an independent set with at least two vertices not in $U_-$.} This corresponds to scenario c) of Statement 2 of Theorem \ref{thm:inducedchar}.

Let $V(H)=\{x_1,x_2,\ldots,x_k\}.$ Without loss of generality, assume $x_1,x_2\in B$ and $x_1,x_2\not\in U_-$.

\textbf{Construction:} define $K := \lceil \frac{n-(k-2)}{2}\rceil$.

\begin{itemize}
    \item Start with $H$
    \item Replace $x_1$ with a large clique on $K$ vertices, call the vertices $\{a_1,\cdots,a_K\}$. These vertices are adjacent to the neighbors of $x_1$.
    \item Replace $x_2$ with a large clique on $K$ vertices, call the vertices $\{b_1,\cdots,b_K\}$. These vertices are adjacent to the neighbors of $x_2$.
    \item Add edges $(a_i,b_i)$ for $1\le i\le K$; in particular, if $i\ne j$, $a_i$ is \emph{not} adjacent to $b_j$. 
\end{itemize}

Note that this graph has $k-2+2K$ vertices.

An illustration of this construction is in Figure \ref{fig6}.

It's straightforward to check that the resulting construction is $(k+1)$-closed, where $k = v(H)$.

\begin{figure}
\[\begin{tikzcd}
	&&&&&&& {a_1} && {b_1} \\
	&&&&&&& {a_2} && {b_2} \\
	{?,x_1} && {?,x_2} &&&&& {a_3} && {b_3} \\
	&&&&&&& {a_4} && {b_4} \\
	& x &&&&&&& x
	\arrow[no head, from=1-8, to=1-10]
	\arrow[no head, from=1-8, to=4-8]
	\arrow[no head, from=1-8, to=5-9]
	\arrow[no head, from=1-10, to=4-10]
	\arrow[no head, from=1-10, to=5-9]
	\arrow[no head, from=2-8, to=2-10]
	\arrow[no head, from=2-8, to=5-9]
	\arrow[no head, from=2-10, to=5-9]
	\arrow[no head, from=3-1, to=5-2]
	\arrow[no head, from=3-8, to=3-10]
	\arrow[no head, from=3-8, to=5-9]
	\arrow[no head, from=3-10, to=5-9]
	\arrow[no head, from=4-8, to=4-10]
	\arrow[no head, from=4-8, to=5-9]
	\arrow[no head, from=4-10, to=5-9]
	\arrow[no head, from=5-2, to=3-3]
\end{tikzcd}\]
\caption{This is an example where the base graph is a wedge, where the $?$'s represent unprescribed vertices. Our construction is on the right.}
\label{fig6}
\end{figure}
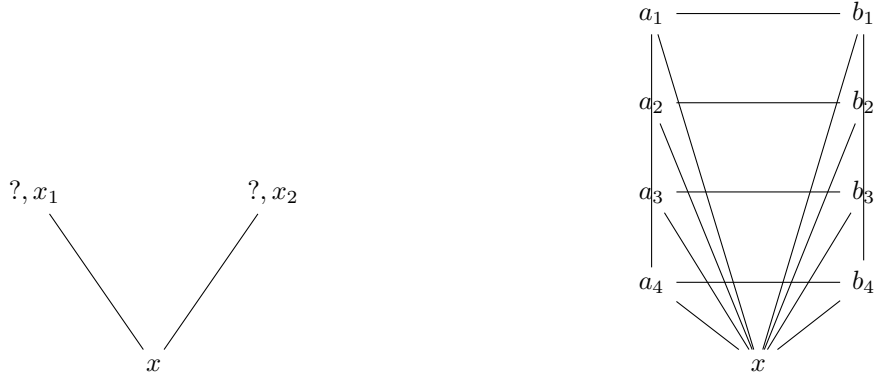

\noindent\textbf{Key claim:} If $S$ is a set of at least $ K+3k^2$ vertices in $\{a_1,\cdots,a_K,b_1,\cdots,b_K\}$ such that for every $i\in[K]$, at least one of $a_i,b_i$ is in $S$, then $S \sqcup (H\setminus \{x_1,x_2\})$ is not an induced $H$-blow-up. 
\begin{proof}

Note that since $B$ contains all the bad vertices in $H$, for all $j\notin B$, we have $|A_j|\le k$ since our construction is $(k+1)$-closed. Thus, there exists at most $k^2$ indices $l$ where both $a_l,b_l$ are in $S$ and one of $a_l,b_l$ belongs to $A_t$ for a non-bad vertex $t$. Hence there exists at least $2k^2$ indices $l$ where both $a_l,b_l$ are in $S$, and belong to $\bigsqcup_{b\in B} A_b $.

We now consider an arbitrary pair $(a_l,b_l)$ in $S$ where $a_l,b_l \in \bigsqcup_{b\in B} A_b$. Since $a_l \sim b_l$ and $H[B]$ is an independent set, it follows $a_l,b_l$ must be put in the same part of the blow-up, i.e. there exists $b_*\in B$ such that $a_l,b_l \in A_{b_*}$. We claim that this implies  $S \subseteq A_{b_*}$. Note that for any $t\ne l$, $a_t$ is adjacent to exactly one of $a_l,b_l$. Since $a_l,b_l\in A_{b_*}$, it follows that if $a_t \in S$, then $a_t\in A_{b_*}$. Similarly we can prove $b_t\in A_{b_*}$ for any $t\neq l$ such that $b_t \in S$.

Remember that our blow-up is $ S \sqcup (H\setminus \{x_1,x_2\})$. Thus, once we put every vertex of $S$ in $A_{b_*}$ there are only $v(H)-2$ vertices remaining. But we need to place them into at least $v(H)-1$ parts,  which leads to a contradiction. This concludes the proof of the key claim. 
\end{proof}

By applying proposition \ref{prop:unifying} on $U = V(H) \setminus \{x_1,x_2\}$, $G$ has at least $2^{-3k^2} (2^K-2)$ maximal $H$-blow-ups.

\

    \noindent\textbf{Case 5: $H$ has exactly one bad twin group $B$, and $H[B]$ is a clique. Moreover, $B\subseteq U_- \sqcup U_+$ and $B\cap U_- \ne \emptyset$. (In particular, this means $B$ is a bad plus twin group.)} This corresponds to scenario a) of Statement 2 of Theorem \ref{thm:inducedchar}.

Define $V(H):= \{x_1,\cdots,x_k\}$. Sometimes, in abuse of notation, we will use $i$ to represent vertex $x_i \in V(H)$. Without loss of generality, let our unique bad plus twin group be $B = \{x_1,\cdots,x_r\}$. We have a clique with vertex set $N(x_1)\sqcup \{x_1\} = \cdots = N(x_r)\sqcup \{x_r\}$, call this $\{x_1,\cdots,x_s\}$. (In particular, this means each vertex in $\{x_{r+1},\cdots,x_s\}$ is adjacent to a vertex outside of $\{x_1,\cdots,x_s\}$). See the following example where the bad plus twin group is $\{x_1\} \subset U_-$.

\[\begin{tikzcd}
	& {x_1} \\
	{x_2} && {x_3} \\
	& {x_4} \\
	& {x_5} \\
	{x_6} && {x_7} \\
	\\
	{x_8} && {x_9}
	\arrow[no head, from=1-2, to=2-1]
	\arrow[no head, from=1-2, to=2-3]
	\arrow[no head, from=2-1, to=2-3]
	\arrow[no head, from=2-1, to=3-2]
	\arrow[no head, from=2-3, to=3-2]
	\arrow[no head, from=3-2, to=4-2]
	\arrow[no head, from=4-2, to=5-1]
	\arrow[no head, from=4-2, to=5-3]
	\arrow[no head, from=5-1, to=5-3]
	\arrow[no head, from=5-1, to=7-1]
	\arrow[no head, from=5-3, to=7-3]
	\arrow[no head, from=7-1, to=7-3]
\end{tikzcd}\]

\textbf{Construction:} 
Let $K := \lceil \frac{n-(k-1)}{2}\rceil$.

\begin{itemize}
    \item start with $H $
    \item Replace $x_1$ with $\{a_1,b_1,a_2,b_2,\cdots,a_K,b_K\}$, add an edge from all of $a_1,b_1,a_2,b_2,\cdots,a_K,b_K$ to $x_2,\cdots,x_s$, where $\{x_2,\cdots,x_s\} = N(x_1)$. 
    \item We blow $x_1$ up into a matching, i.e. $$E(G[\{a_1,b_1,a_2,b_2,\cdots,a_K,b_K\}]) = \{ (a_1,b_1), (a_2,b_2), \cdots, (a_K,b_K)\},$$ where $G$ is the new graph we construct.
\end{itemize}

Note this graph has $k-1+2K$ vertices.

Here is an example for the above graph when $K=3$:

\[\begin{tikzcd}
	{a_1} &&&& {b_1} \\
	{a_2} &&&& {b_2} \\
	& {a_3} && {b_3} \\
	\\
	& {x_2} && {x_3} \\
	&& {x_4} \\
	&& {x_5} \\
	& {x_6} && {x_7} \\
	\\
	& {x_8} && {x_9}
	\arrow[no head, from=1-1, to=1-5]
	\arrow[no head, from=1-1, to=5-2]
	\arrow[no head, from=1-1, to=5-4]
	\arrow[no head, from=1-5, to=5-2]
	\arrow[no head, from=1-5, to=5-4]
	\arrow[no head, from=2-1, to=2-5]
	\arrow[no head, from=2-1, to=5-2]
	\arrow[no head, from=2-1, to=5-4]
	\arrow[no head, from=2-5, to=5-2]
	\arrow[no head, from=2-5, to=5-4]
	\arrow[no head, from=3-2, to=3-4]
	\arrow[no head, from=3-2, to=5-4]
	\arrow[no head, from=3-4, to=5-2]
	\arrow[no head, from=5-2, to=3-2]
	\arrow[no head, from=5-2, to=5-4]
	\arrow[no head, from=5-2, to=6-3]
	\arrow[no head, from=5-4, to=3-4]
	\arrow[no head, from=5-4, to=6-3]
	\arrow[no head, from=6-3, to=7-3]
	\arrow[no head, from=7-3, to=8-2]
	\arrow[no head, from=7-3, to=8-4]
	\arrow[no head, from=8-2, to=8-4]
	\arrow[no head, from=8-2, to=10-2]
	\arrow[no head, from=8-4, to=10-4]
	\arrow[no head, from=10-2, to=10-4]
\end{tikzcd}\]

Note that if $S\subsetneq \{1,\cdots,K\}$ is a nonempty subset, then $\{a_i \colon i\in S\} \sqcup \{b_i \colon i\notin S\} \sqcup \{x_2,\cdots,x_k\}$ is a $(U_+,U_-)$-prescribed induced $H$-blow-up. It is a $(U_+,U_-)$-prescribed induced $H$-blow-up by setting $A_1 = \{a_i \colon i\in S\} \sqcup \{b_i \colon i\notin S\}, A_j = \{x_j\}$ for $2\le j\le k$. 

\noindent\textbf{Key claim:} If $S$ is a set of at least $ K+3k^2$ vertices in $\{a_1,\cdots,a_K,b_1,\cdots,b_K\}$ such that for every $i$, at least one of $a_i,b_i$ is in $S$, then $S \sqcup \{x_2,\cdots,x_k\}$ is not an induced $H$-blow-up. 

\begin{proof}  Similar to Case 4, there are at least $2k^2$ indices $l$ where both $a_l,b_l$ are in $S\cap\bigsqcup_{b\in B} A_b $. Let $a_{l_1},a_{l_2},b_{l_1},b_{l_2}\in S\cap\bigsqcup_{b\in B} A_b.$ Since $a_{l_1} \not\sim a_{l_2}$ and $H[B]$ is a clique, it follows $a_{l_1},a_{l_2}$ must be put in the same part of the blow-up, i.e. there exists $b_*\in B$ such that $a_{l_1},a_{l_2} \in A_{b_*}$. Similarly, since $b_{l_1}\not\sim a_{l_2}$, it follows that $b_{l_1}\in A_{b_*}$. But this implies $A_{b_*}$ is neither a clique nor an independent set, contradicting $b_*\in B\subseteq U_- \sqcup U_+$. 
\end{proof}
By applying proposition \ref{prop:unifying} on $U = V(H) \setminus \{x_1\}$, $G$ has at least $2^{-3k^2} (2^K-2)$ maximal $H$-blow-ups.

\

    \noindent\textbf{Case 6: $H$ has exactly one bad twin group $B$, and $H[B]$ is independent set. Moreover, $B\subseteq U_- \sqcup U_+$ and $B\cap U_- \ne \emptyset$.} This corresponds to scenario b) of Statement 2 of Theorem \ref{thm:inducedchar}.

\textbf{Construction:}

We consider the same construction as Case 5.

\noindent\textbf{Key claim:} If $S$ is a set of $\ge K+3k^2$ vertices in $\{a_1,\cdots,a_K,b_1,\cdots,b_K\}$ such that for every $i$, at least one of $a_i,b_i$ is in $S$, then $S \sqcup \{x_2,\cdots,x_k\}$ is not an induced $H$-blow-up.

\begin{proof}
Similar to Case 4, there exists at least $2k^2$ indices $l$ where both $a_l,b_l$ are in $S$, and belong to $\bigsqcup_{b\in B} A_b $. Moreover, for each pair $(a_l,b_l)\in S\cap \bigsqcup_{b\in B} A_b$, there exists $b'_l\in B$ such that $a_l,b_l\in A_{b'_l}$. Since every vertex in $B$ is prescribed, each $b'_l$ is distinct. Hence $|B|\geq 2k^2>k$, a contradiction.
\end{proof}
By applying proposition \ref{prop:unifying} on $U = V(H) \setminus \{x_1\}$, $G$ has at least $2^{-3k^2} (2^K-2)$ maximal $H$-blow-ups. 
\end{proof}

\begin{corollary}\label{cor: finiteFamilyInduced}
 Let $\mathcal{H} = \{ (H_1, U_{1,+}, U_{1,-}), (H_2, U_{2,+}, U_{2,-}),\cdots, (H_m, U_{m,+}, U_{m,-})\}$. Let $c> \max_i v(H_i) + 2$. If there exists $i$ such that $(H_i,U_{i,+},U_{i,-})$ is in Case 2 of theorem \ref{thm:inducedchar} then $\mathcal{H}$ is a $c$-closed induced exponential family. Otherwise, $(H_i,U_{i,+},U_{i,-})$ is in Case 2 of theorem \ref{thm:inducedchar} for all $i$, and $\mathcal{H}$ is a $c$-closed induced polynomial family.
\end{corollary}

\section{Counting Maximal Blow-ups for Patterns in an Infinite Family $\mathcal{H}$}
\subsection{Challenges for Infinite Families}
Recall the definitions (same definitions from Definitions \ref{def: polyexpinf} and \ref{def: familyinduced}): 
\begin{definition}
Let $\mathcal{H} = \{(H_1,U_1), (H_2,U_2), \dots\}$ be a family of graphs, where $H_i$ is a simple graph and $U_i \subseteq V(H_i)$ is the set of vertices in $H_i$ that are prescribed to be cliques. We say $\mathcal{H}$ is a \textbf{$c$-closed polynomial pattern family}, if there is a polynomial $f(x)\in \mathbb R[x]$ such that for an arbitrary $c$-closed graph $G$ on $n$ vertices,  $$\# \{S\subset V(G) \colon \text{ there exists } i \text{ such that } G[S] \text{ is a maximal } H_i\text{-blow-up}\} \le f(n).$$ We say $\mathcal{H}$ is a \textbf{$c$-closed exponential pattern family}, if there exist a constant $\epsilon>0$ and a sequence of $c$-closed graphs $(G_n)_{n\ge 1}$  such that $v(G_n)\to \infty$, and for all $n$, $$\# \{S\subset V(G_n) \colon \text{ there exists } i \text{ such that } G_n[S] \text{ is a maximal } H_i\text{-blow-up}\} \ge e^{v(G_n)^\epsilon}.$$ 
\end{definition}

One can also easily define it analogously for induced patterns (see Definition \ref{def: familyinduced}). 



There are many difficulties in applying our insights for the solution of this problem to one graph to characterize in general whether a family is exponential or polynomial. Suppose we have a family $\{(H_i,U_i)\}_{i\ge 1}$ with $v(H_i) \to \infty$. If we apply Theorem \ref{thm: noninduced finite}, then we can see that there exists a graph $G$ on $n$ vertices (assume $n$ very large compared to $c$) with at most $(2n^{\max\{c-1,1\}})^{v(H_i)}$ maximal blow-ups. However, as $v(H_i) \to \infty$, this does not immediately allow us to conclude that there is a polynomial $f \in \mathbb{R}[x]$ such that for an arbitrary $c$-closed graph $G$ on $n$ vertices,  $$\# \{S\subset V(G) \colon \text{ there exists } i \text{ such that } G[S] \text{ is a maximal } H_i\text{-blow-up}\} \le f(n).$$

 Another difficulty is that suppose we can have a sequence $c$-closed exponential pattern families $\mathcal{F}_1, \mathcal{F}_2,\cdots$ and we might end up use these to construct a family that is not a $c$-closed exponential pattern family as follows: we choose $F_{ki_k} \in \mathcal{F}_k$ such that for any $G$, the number of maximal $F_{ki_k}$-blow-ups of $G$ is $\le \exp( v(G)^{2D_i})$. If $D_i \to 0$, then the artificially constructed family $\{F_{k,i_k}\}_{k\ge 1}$ is not exponential.

\subsection{Infinite Families of Bounded Degree}\label{sec: bounded}
In this subsection, we prove the following general theorem (restatement of Theorem \ref{thm:infinitebd}). 
\begin{theorem}
    For any  $d > 1$, there exists $c > 0$ such that if each pattern $H$ in $\mathcal{H}$ has maximum degree bounded above by $d$, and $\mathcal{H}$ is an infinite family, then $\mathcal{H}$ is a $c$-closed exponential pattern family. 
\end{theorem}
This assumption of bounded degree is natural, for example, when $\mathcal{H}$ consists of infinitely many cycles, or consists of infinite many of paths, or consists of infinitely many binary trees. 
This result is an easy corollary of the following theorem (theorem \ref{thm: bounded}). 
\begin{theorem}
    Let $d\ge 2$ be a positive integer. Let $H$ be a connected graph with maximum degree $d$ and $N$ vertices, where $N > d^{d^{d^d}}$. Then there exists a graph $G$ such that $N \leq |V(G)| \leq 2^{d^{21d+20}}N$, the maximum degree is at most $d2^{d^{21d+20}}$, and there are at least $2^{|V(G)|/2^{d^{21d+20}}}$ vertex sets in $G$ corresponding to maximal blow-ups of $H$.   
\end{theorem}

Note that no attempt is made to optimize the constants.

\begin{proof}
 Let the minimum degree in $H$ be $\delta$. For each $\delta \leq j \leq d$, let $N_j$ be the number of vertices in $H$ of degree $j$. We choose the minimal $i$ such that $N_i > \frac{N}{d^{20(d-i+1)-1}}$ and it is easy to see such an $i$ exists by the fact that $\sum_{i=\delta}^d \frac{N}{d^{20(d-i+1)-1}} < N$. In particular, this means that at least \[\frac{N}{d^{20(d-i+1)-1}}\left(1-\frac{2}{d^{16}}\right) \ge \frac{N}{2d^{20(d-i+1)-1}}\] vertices $v$ of degree $i$ such that  $\text{dist}(v,v') \ge 5$ for all $v'\in V(H), \deg(v')<i$. This is because for all $j=\delta,\cdots,i-1$, there are at most $\frac{N}{d^{20(d-j+1)-1}} $ vertices of degree $j$, so there are a total of at most $\frac{2N}{d^{20(d-i+2)-1}}$ vertices of degree less than $i$. Thus, the number of vertices within the $4$-neighborhood of a vertex of degree at most $i$ is at most $\frac{2N}{d^{20(d-i+2)-1}} d^4 = \frac{2N}{d^{20(d-i+1)-1}} d^{-16}$.

Recall we use $[v]$ to denote the twin group of $v$, which is defined to be the set containing all twins of $v$ and $v$ itself. Notice each twin group is either a clique or an independent set, and vertices in the same twin group must have the same degree. 

Let $n_{i^*}$ denote the number of twin groups of vertices in degree $i^*$. If a vertex is adjacent to one vertex in a twin group, then it is adjacent to all vertices of that twin group. Since the maximum degree of $H$ is $d$, there can be no more than $d$ vertices in a twin group. Therefore 
\begin{equation}
    n_{i^*} \geq \frac{N_{i^*}}{d} \ge \frac{N}{2d^{20(d-i^*+1)}} . \label{eq:twingroup1}
\end{equation}

  Given a twin group $[w]$, we denote $N_{r}([w])$ as the set of vertices of distance $r$ to $w$ in $H$ and are not in $[w]$ (note this does not depend on the representative $w\in [w]$). We denote $N_{\leq r}([w])$ as the set of vertices of distance at most $r$ from $w$ in $H$ and are not in $[w]$. The \textbf{degree of a twin group} is just the degree of a vertex in the twin group.

\

\subsubsection{Construction of the host graph $G$. }

 We have just shown that there are at least $\frac{N}{2d^{20(d-i^*+1)}}$ twin groups $[v]$ of degree $i$ such that $N_{\le 4}([v])$ does not contain any vertices of degree smaller than $i^*$. 

We choose twin groups of degree $i^*$ such that any two chosen vertices are at least distance $11$ away until we have chosen exactly $T := \frac{N}{2d^{20d-20i^*+31}}$ twin groups. \footnote{Technically the number of twin groups is an integer, however an error of 1 does not affect future arguments and bounds. } Since choosing each vertex prevents us from choosing at most $d^{11}$ vertices, this is always doable.

    For each chosen twin group $[x]$, we distinguish the following two cases.
    \begin{enumerate}
        \item We say $[x]$ is of {\it Type I} if at least one of the following holds. \begin{enumerate}
            
            \item If there is some vertex $v$ in $N_{\leq 2}([x]) \cup [x]$ of degree $i$ and there is $u \in N(v)\setminus [v]$ with $\deg(u)=i^*$. 
            \item If we are not in 1(a), so in particular, all vertices in $N(x)\setminus [x]$ have degree greater than $i^*$. Furthermore, there is a neighbor of $x$ not in $[x]$ which is not adjacent to any vertex of degree $i^*$ not in $[x]$. 
        \end{enumerate}
        \item We say $[x]$ is of {\it Type II} if $[x]$ is not of Type I. In other words, every vertex in $N_{\leq 2}([x]) \cup [x]$ has degree at least $i^*$, and since $[x]$ is not of Type I (a), for every vertex $w$ in $N_{\leq 2} ([x])\cup [x]$ of degree $i^*$, all of its neighbors not in $[w]$ have degree larger than $i^*$. Note $[x]$ is allowed to be a clique.  
        
    \end{enumerate}

One observation is that note that since $H$ is connected, $N(x) \setminus [x]$ is always nonempty. The reason is follows. If $N(x) \subseteq [x]$, then observe that for any $x'\in N(x)$, $N(x') \sqcup \{x'\} = N(x) \sqcup \{x\}$. Thus, $H\setminus (N(x)\sqcup \{x\}), N(x)\sqcup \{x\}$ forms a disconnection of $H$, since $|H| \ge d^{d^{d^d}}$ and $|N(x)\sqcup \{x\}| \le d+1$. 

We have three scenarios on the construction:

\begin{itemize}
    \item Case 1 (a): There are at least $\frac T3$ chosen twin groups of Type I (a).
    \item Case 1 (b): We are not in Case 1 (a), and there are at least $\frac T3$ chosen twin groups of Type I (b). 
    \item Case 2: We are not in Case 1 (a) or (b), so there are at least $\frac T3$ chosen twin groups of Type II.
\end{itemize}

In each case, the construction of $G$ is as follows by applying the following modification on $H$ for the $\frac T3$ chosen twin groups $[x]$ of the corresponding type. (From now on, when we say a {\it chosen twin group}, we mean one of these $\frac T3$ chosen ones). For each such twin group $[x]$, we choose an arbitrary representative $x\in [x]$, two vertices $v,u$ with small distance from $[x]$, and modify $H$ as follows: 
    \begin{enumerate}
        \item If we are in Case 1 (a), then $x$ is Type I (a). Let $v$ be a vertex in $N_{\leq 2}(x) \cup [x]$ of degree $i^*$ which has a neighbor $u$ not in $[v]$ with degree $i^*$. If we are in Case 1 (b) but not Case 1 (a), then $x$ is of Type I(b), let $v$ be $x$, and choose $u \in N(x) \setminus [x]$ such that $u$ is not adjacent to any vertex of degree $i^*$ other than vertices in $[x]$. 
 
        In these cases,  after fixing the vertices $v, u$, 
    make $C_d-1$ copies of $[v']$ and $[u']$ for $[v]$, $[u]$ respectively, where $C_d = O(1)$ is a constant only depending on $d$ and the type. We will specify the value of $C_d$ later. 
    
    Specifically, if $[v] = \{v_1,\cdots,v_a\}$ and $[u] = \{u_1,\cdots,u_b\}$ this means:
    
    \begin{itemize}
        \item We add vertices $(v_{j,l})_{2\le j\le C_d,1\le l\le a}$ and $(u_{j,l'})_{2\le j\le C_d,1\le l'\le b}$ into $G$. Here we can think of $v_{1,l}$ and $u_{1,l'}$ as vertices originally in $H$. 
        
        \item We add an edge between $v_{j,l},u_{j,l'}$ for all $1\le j\le C_d, 1\le l\le a, 1\le l'\le b$. (Note $v_{j,l}$ is not adjacent to $ u_{j',l'}$ for all $l,l'$ and $j\ne j'$.) 
        
        \item If $[v]$ is a clique, then $\{ v_{j,1},\cdots, v_{j,a}\}$ is a clique; otherwise $\{ v_{j,1},\cdots, v_{j,a}\}$ is an independent set. Also, if $j\ne j'$, then $v_{j,l}$ is not adjacent to $v_{j',l'}$ for all $l,l'$. We perform the analogous construction on $u$.  
        
       \item If $w\notin [v]\cup [u]$ is a neighbor of $v$, then $w$ is adjacent to $v_{j,l}$ for all $1\le j\le C_d, 1\le l\le a$. Analogous comment for $u$.   
    \end{itemize} 
    
    Here $C_d=2$ if $x$ is of Type I (b). For Type I (a), the value of $C_d$ will be specified later. 
    
        \item If we are in Case 2, i.e. $[x]$ is of Type II, then $x$ is of degree $i^*$, and all of its neighbors not in $[x]$ have degree at least $i^*+1$. Let $v$ be $x$ and let $u$ be its arbitrary neighbor not in $[x]$. Note by the assumption that $x$ is of Type II, $u$ has some neighbor of degree $i$ not in $[x]$. Let $S$ be the set of neighbors of $u$ not in $[x]$ with degree $i^*$. Thus $S \neq \emptyset$. 

        In this case, after fixing the vertices $v, u$, 
    add a copy $[v']$ and $[u']$ for $[v]$,  $[u]$ respectively and add a copy $S'$ for $S$. Add edges so that $[v']$ is adjacent to $[u']$ and $N_1([v])\setminus [u]$. And add edges so that $[u']$ is adjacent to the neighbors of $[u]$ that not in $[u] \cup S\cup [v]$,  and add edges between $[u']$ and every vertex in $S'$. Finally, add edges between each vertex $s' \in S'$ to the set of vertices which are adjacent to the vertex $s \in S$ that it is a copy of except not adding edges between $S'$ and $[u]$. In case 2, we define $C_d:=2$.
    \end{enumerate}

\subsubsection{Notations and Properties of the Construction }
   We now define some notions that will be used frequently in the rest of the proof. 
   
   In case 1 (b) and case 2, call $[v]$ and $[v']$ {\it top groups} and $[u]$ and $[u']$ {\it special neighbor groups} (of $[v]$ and $[v']$ respectively); and call $u$ {\it special neighbor} of $v$, and similarly for $u'$. Call vertices in $[v]$ and $[v']$ \emph{top vertices}. Note that each top vertex has degree at most $i^*$. Call $S, S'$ the {\it special second neighbors} (of $[v]$ and $[v']$ respectively). This admits an obvious generalization to arbitrary $C_d$ in case 1 (a). 
   
    A {\it top block} is defined as follows: if $[x]$ is of Type I, the top block that $[x]$ is in is $G[[v] \cup [v']\cup N_{\leq 2}([v])\cup N_{\leq 2}([v'])]$ (containing $[u]$ and $[u']$); if $[x]$ is of Type II, it is the graph induced by $[v]$, $[v']$, $N_{\leq 3}([v])$, and $N_{\leq 3}[v']$ (containing $[u]$ and $[u']$, $S, S'$). We say the top block constructed above based on $[x]$ are  of {\it Type I (a), (b)}, and {\it Type II} respectively. By construction, each top block corresponds to a chosen twin group, hence the number of top blocks in $G$ is $\frac T3$. 
    
    Note that initially we chose twin groups to have distance at least $11$ from each other; a subtlety is that the top vertices are not necessarily in the twin group, but have distance at most 2 from the twin group, so this guarantees that the top vertices of the twin group are of distance at least $7$ from each other, we can see that the top blocks are disjoint. Furthermore, recall that in every top block, the minimum degree is equal to $i$. 

    \begin{figure}
        \centering
        \includegraphics[scale=0.4]{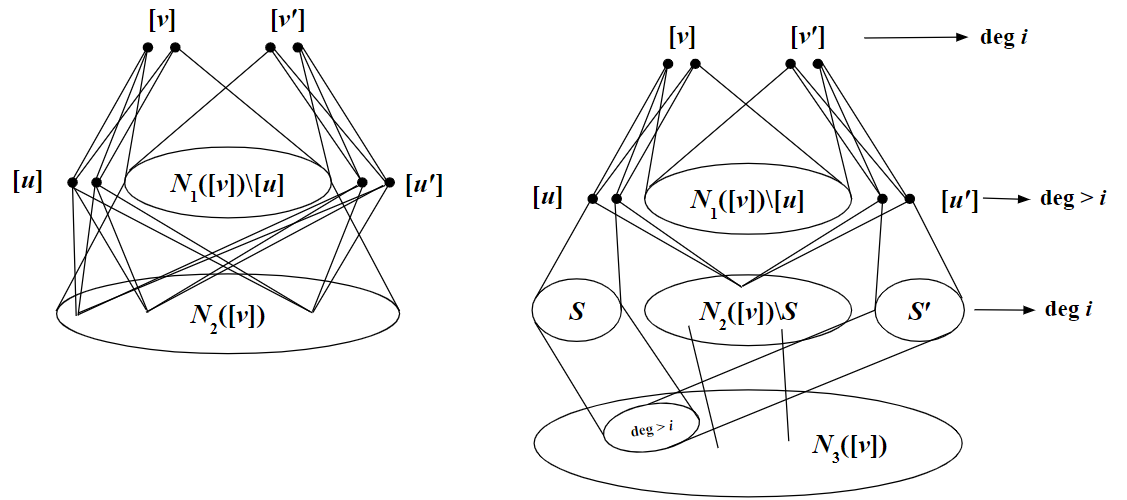}
        \caption{Construction of top blocks of Type I (b) (left) and Type II (right) in $G$. \footnotesize{In Type I (b): $v$ has virtual degree $i$ and $u$ has virtual degree larger than $i$, and $u$ is not adjacent to any vertex of virtual degree $i$ except vertices in $[v]$. In Type II block, every vertex has virtual degree at least $i$; $S$ and $[v]$ are the only neighbors of $[u]$ which are of virtual degree $i$; neighbors of vertices in $S$ are of virtual degree larger than $i$.} }
        \label{fig:topblock}
    \end{figure}

        For any vertex in $G$ that was originally in $H$ we refer to its {\it virtual degree} as its degree in $H$. For any vertices in a top group $[v']$ or a special neighbor group $[u']$ or a special second neighbor $S'$, their virtual degree will be the degree in $H$ of the vertex they are a copy of.
An illustration of the top blocks of Types I (b) and II are shown in Figure \ref{fig:topblock}.

The next claim follows immediately from the construction.

    \begin{claim}\label{claim:deg1} If $w \in G$ is adjacent to a top vertex or a special neighbor or a special second neighbor, and itself is not a top vertex, or a special neighbor, or a special second neighbor, then $w$ has larger degree in $G$ than its virtual degree. Otherwise, $w$ has the same degree in $G$ as its virtual degree.

    \end{claim}

        \begin{claim}\label{claim:deg2}
     In $G$, there is no vertex in a  top block of Type II or Type I (b) which is of virtual degree $i^*$ but is of degree at least $i^*+1$ in $G$.
    \end{claim}
    \begin{proof}
        In a top block of Type II, the top vertices form a twin group with vertices of degree $i^*$. Define $v,v',u,u',S,S'$ as in the setup. By Claim \ref{claim:deg1}, the only vertices $w$ of virtual degree $i^*$ but a larger degree in $G$ must be adjacent to one of $v,v',u,u'$ or some vertex in $S$ or $S'$. Note that $v,v'$ or vertices in $S$ or $S'$ have degree $i^*$, so if $w$ is adjacent to one of these, say $v$, then by definition of Type II, $w$ must be in $[v]$, so it is a copy of some vertex in $H$ and its virtual degree is equal to its degree in $G$. Otherwise, $w$ is adjacent to $u$ or $u'$. Since $w$ has virtual degree $i^*$, from the definitions we can see that $w\in S \cup S' \cup[v] \cup [v']$ by the definition of $S$. But every vertex in this set has degree $i^*$ in $G$ by Claim \ref{claim:deg1}, 
        a contradiction.

        In a top block of Type I (b), let $v$ be a top vertex. Then all vertices in $N_{G}(v) \setminus [v]$ must have virtual degree $>i^*$. The same holds for $v'$. If $u$ is a top neighbor, then all vertices in $N_{G}(u) \setminus [u]$ must have virtual degree $>i^*$. Thus, there are no vertices in a top block of Type I (b) with virtual degree $i^*$ and degree $>i^*$ in $G$. 
    \end{proof}

    Recall $\frac T3 = \frac{N}{6d^{20d-20i^*+31}}$ is the number of chosen twin groups, which is equal to the number of top blocks in $G$. We now create $(C_d)^{T/3}$ blow-ups (not necessarily maximal) of $H$ in $G$, call them {\it naive blow-ups}  (Recall $C_d=2$ unless we are in Type I (a), in which case we have not yet specified the value of $C_d$). In each top block, we choose everything in the block that is not $[v],[u],[v'],[u']$ and we are then also choosing either $[v]$ and $[u]$ (and $S$ if the block of Type II), or choosing $[v']$ and $[u']$ (and $S'$ if the block of Type II). In the case of Type I (a), it's the same idea, just a more complicated expression. If $[v] = \{v_1,\cdots,v_a\}$ correspond to vertices $(v_{j,l})_{1\le j\le C_d, 1\le l\le a}$ in the top block and $[u] = \{u_1,\cdots,u_b\}$ correspond to vertices $(u_{j,l'})_{1\le j\le C_d, 1\le l'\le b}$ in the top block, then we choose one $j\in C_d$ and include $v_{j,l}, u_{j,l'}$ for all $l$ and $l'$. In all cases, we also keep all the vertices from $H$ that are not in a top group, special neighbor group, or a special second neighbor. This gives us $C_d^{T/3}$ naive blow-ups; note that each naive blow-up is an induced copy of $H$. 

  Throughout the rest of the proof, if $V(B) \subset V(G)$ is a blow-up, then we will use \[\phi: V(B) \to V(H)\] to denote the vertex partition of a blow-up $B$. Two vertices in $w,w'\in V(B)$ satisfy $\phi(w)=\phi(w')$ if and only if they have been grouped together in $B$ to correspond to the same vertex in $H$. 

Note that when we say vertices of degree/virtual degree/twin group $i^*$ in $V(B) \subseteq V(G)$, we are referring to the degree/virtual degree/twin group in $G$.

\subsubsection{Type I (b) and Type II} 

We have already created $2^{T/3}$ naive blow-ups, which are not necessarily maximal. The idea is to analyze twin groups of virtual degree $i^*$ in $G$.

\begin{claim}\label{claim:degreedelta}
Let $w,t\in V(B)$ such that the virtual degree of $w$ and $t$ are both equal to $i^*$, and $[\phi(w)]_H=[\phi(t)]_H$ satisfy $\deg_H(\phi(w))=i^*$. Then $[w]_{G} = [t]_{G}$.

\end{claim}
\begin{proof}

By Claim \ref{claim:deg2}, since $w$ has virtual degree $i^*$, $\deg_{G}(w) = \deg_{G}(t) = i^*$. 

 Let $w,t$ be given as in the claim. Let $b_1,\cdots,b_{i^*}$ be pairwise distinct vertices such that $  N_H(\phi(w)) = \{ b_1,\cdots,b_{i^*}\}, $ and let $a_1,\cdots,a_{i^*}$ be the neighbors of $w$ in $G$. Note that $\phi(a_1),\cdots,\phi(a_{i^*})$ is necessarily a bijection of $b_1,\cdots,b_{i^*}$. 

 Case 1: $[\phi(t)]_H$ is an independent set. Then $N_H(\phi(t)) = N_H(\phi(w)) = \{b_1,\cdots,b_{i^*}\}$. Since $[\phi(t)]_H = [\phi(w)]_H$, $t$ must be adjacent to all of $a_1,\cdots,a_{i^*}$. Since $t$ has degree $i$, $N_{G}(t) = \{a_1,\cdots,a_{i^*}\} = N_{G}(w)$, which implies that $[t]=[w]$.

Case 2: $[\phi(t)]_H$ is a clique. In this case, $N_H(\phi(t))$ is not equal to $N_H(\phi(w))$ unless $\phi(t) = \phi(w)$ as vertices. If $\phi(t)=\phi(w)$, then again $N_H(\phi(t))=N_H(\phi(w))$ and we finish as in Case 1. Otherwise, we define $\phi(w) = b_0, \phi(t) = b_j$ for some $j\ne 0$, so $N_H(b_0) = \{b_1,\cdots,b_{i^*}\}$ and $N_H(b_j) = \{b_0,\cdots,b_{i^*}\} \setminus \{b_j\}$. We relabel the $a_j$'s so that $\phi(a_j)=b_j$ for all $j=1,\cdots,i^*$. Since $t$ is a neighbor of $w$, $t=a_j$. This means that the neighbors of $t$ are $w,a_1,\cdots,a_{j-1},a_{j+1},\cdots,a_k$, so in particular, $[t]_{G} = [w]_{G}$.
\end{proof}

  We now make a simple observation.
\begin{claim}\label{claim:smallphi}
    For any vertex $z \in V(H)$, the pre-image $\phi^{-1}(z)$ in $G$ has size at most $2d$.
\end{claim}
\begin{proof}
  Since $H$ has no isolated vertices, $\phi^{-1}(z)$ is contained in a neighborhood of some vertex in $G$. Since $G$ has maximum degree at most $2d$, $|\phi^{-1}(z)| \leq 2d$.
\end{proof}

The key claim now is that any blow-up must contain relatively few twin groups of degree $i^*$ (which is equivalent to virtual degree $i^*$). Recall $n_{i^*}$ is the number of twin groups of degree $i^*$ in $H$, and when we say degree we mean degree in $G$. 

\begin{claim}\label{claim:smallnum}
    There are at most $n_{i^*} + 2d \cdot \frac{2N}{d^{20(d-i^*+2)-1}}$ twin groups of virtual degree $i^*$ in $G$ with at least one vertex in $B$.
\end{claim}

\begin{proof}
    Let $v$ be a vertex of virtual degree $i^*$ in $B$. Then either $\deg_H(\phi(v)) = i^*$ or $\deg_H(\phi(v)) < i^*$. We say a twin group $[x]$ in $G$ with virtual degree $i^*$ and intersect nontrivially with $B$ is \emph{ok} if there exists $z\in [x]_{G}\cap B$ such that $\deg_H(\phi(z))=i^*$; otherwise, every $z\in [x]_{G}\cap B$ must satisfy $\deg_H(\phi(z))<i^*$, and we say the twin group is {\it not ok}.  

    Cases with $\deg_H(\phi(v)) = i^*$: by claim \ref{claim:degreedelta}, if $[x]_H$ is a twin group of degree $i^*$, then the vertices of virtual degree $i^*$ in $\phi^{-1}([x]_H) $ must be in the same twin group in $G$. Thus, there exist at most $n_i$ ok twin groups.

    Cases with $\deg_H(\phi(v)) < i^*$: by claim \ref{claim:smallphi}, the fibers of $\phi$ have size at most $2d$. Thus, $|\phi^{-1}(\{ y\in H \colon \deg_H(y)<i^*\})| \le 2d |\{ y\in H \colon \deg_H(y)<i^*\}| \le 2d\frac{2N}{d^{20(d-i^*+2)-1}}$. Thus there exists at most $2d\frac{2N}{d^{20(d-i^*+2)-1}}$ not ok twin groups. 
    
    Adding the two cases yields the desired result. 
\end{proof}

Recall \[U:=\frac T3= \frac{N}{6d^{20d-20i^*+31}}\] is the number of twin groups that we chose, i.e. the number of top blocks. Consider maximal blow-ups containing at least one naive blow-up. Such a maximal blow-up must contain all twin groups that are not top vertices or special neighbors or special second neighbors. Let $X$ be the number of top groups or special second neighbor groups contained in a fixed copy of $H$ (in $G$) with degree $i^*$.  There are exactly $n_{i^*} - X$ twin groups of degree $i^*$ in $G$ that are not top groups or special second neighbor groups. Thus, by claim \ref{claim:smallnum}, $B$ contains at most $X + \frac{4Nd^2}{d^{20(d-i^*+2)}}$ top groups or twin groups of special second neighbors of degree $i^*$. Observe the naive blow-ups all differ in the set of twin groups of virtual degree $i^*$ that arise from the top vertices. For the two copies of each top group or each special second neighbor group, every naive blow-up contains exactly one of them. Thus, the blow-up needs to contain at least one copy of every top group or each special second neighbor group, and there are at most $\frac{4Nd^2}{d^{20(d-i+2)}} = 24Ud^{-7}$ top groups or special second neighbor groups such that the blow-up contains both copies. Thus, the number of naive blow-up it contains is at most \[ 2^{24Ud^{-7}} = 2^{\frac{4Nd^2}{d^{20(d-i+2)}}}.\]

We construct an arbitrary map 
\[f \colon \{ \text{Vertex sets that form naive blow-ups} \} \to  \{ \text{Vertex sets that form maximal blow-ups} \}\]
such that $f(B) \supseteq B$ for all $B$. We just showed the fibers of this map has at most $2^{24Ud^{-7}}$ elements, so the image of this map has at least $2^{U}/2^{24Ud^{-7}} = 2^{U(1-24d^{-7})}$ elements. Thus, the number of maximal blow-ups is at least $$2^{U(1-24d^{-7})} \ge 2^{U/2} = 2^{N/(12d^{20d-20i+31})} \ge 2^{v(G)/(d^{20d+20})}. $$

\subsubsection{Type I (a).}

In this case, recall in the construction of $G$ there are $C_d$ copies for $\frac T3$ top groups of Type I (a) with top degree $i^*$ and vertices these top groups have distance at least $5$ away from any vertex of degree less than $i^*$.

The key is that we have a bound on the number of vertices of virtual degree $i^*$ in the blow-up that does not depend on the choice of $C_d$. This is very useful because we have $C_d^{T/3}$ naive blow-ups, each of which are an induced copy of $H$, and this allows up to show that any blow-up cannot contain that many naive blow-ups.

Let $A$ denote a blow-up of $H$ contained in $G$ that contains at least one naive blow-up. We divide the vertices of virtual degree $i^*$ in $A$ into three cases:

\begin{itemize}
    \item Case I: $\deg_H(\phi(x)) < i^*$
    \item Case II: $\deg_H(\phi(x)) = i^*$
    \item Case III: $\deg_H(\phi(x)) > i^*$
\end{itemize}

The first case, where $\deg_H(\phi(v))<i^*$, is the most involved. We first handle Case I. The crux is the following claim.

\begin{claim}\label{claim:topdeg}
    There is no vertex $w\in A$ that is a copy of some vertex $w'\in V(H)$ in the construction of $G$ such that $\deg_H(\phi(w)) < i^*$. 
\end{claim}

\begin{proof}
    Let $\pi \colon V(G)\to V(H)$ be the canonical projection map; since every vertex in $G$ is either originally in $H$ or copy of some vertex in $H$, $\pi(y)$ just denote which vertex in $H$ it is a copy of. Let $X := \{ x\in V(G)\colon \deg_{G}(x)<i^*\}$. Note these vertices have degree equal to virtual degree. Note that for $x\in X$, $\pi^{-1}(\{x\})$ has size $1$, so it makes sense to view $\pi^{-1}(x)$ as an element of $G$. Let $Y := \{ y\in A \colon \deg_H(\phi(y))<i^*\}$. Note that $Y \supseteq X$, since  we assume that the blow-up contains at least one naive blow-up, and every naive blow-up already contains all of $X$. 
    
    We define a map \[\tilde{\phi} \colon Y \to X, \tilde{\phi} = \pi^{-1} \circ \phi.
    \]

    We have a few basic observations:

    Note for $x,y\in Y$, $\tilde{\phi}(x)= \tilde{\phi}(y) \iff \phi(x)=\phi(y)$, so $x,y$ must have a common neighbor, so $\text{dist}(x,y) \le 2$. If $\phi(x)\ne \phi(y)$, then $\text{dist}(x,y) \le \text{dist}(\phi(x),\phi(y))$ because if $\text{dist}(\phi(x),\phi(y))=D$, then there is a path $\phi(x),a_1,\cdots,a_{D-1}, \phi(y)$ in $H$. There exists $y_1,\cdots,y_{D-1} \in V(G)$ such that $\phi(y_{i^*}) = a_{i^*}$ for all $i^*\in \{1,\cdots,D-1\}$. Clearly, $x,y_1,\cdots,y_{D-1},y$ is a path of length $D$ in $G$. Now observe that if $a,b\in V(H)$ with $\deg(a),\deg(b)<i$, then $\text{dist}(a,b) = \text{dist}(\pi^{-1}(a),\pi^{-1}(b))$.
    
    Here is the main step: there must exists some $t\ge 1$ such that $im(\tilde{\phi}^t) = im(\tilde{\phi}^{t+1})$, where $\tilde{\phi}^t := \underbrace{\tilde{\phi} \circ \cdots \circ \tilde{\phi}}_{t \text{ times }}$. Consequently, $im(\tilde{\phi}^t) = im(\tilde{\phi}^{t+1}) = \cdots = im(\tilde{\phi}^{2t})$.

    Now, suppose $y$ is a copy of some vertex in $H$, and $\deg_H(\phi(y)) <i^*$. Then  $\deg_{G}(\tilde{\phi}(y)) = \deg_H(\phi(y)).$  Since $im(\tilde{\phi}^{t})=im(\tilde{\phi}^{2t})$, there exists $z\in X$ such that $\tilde{\phi}^{t}(y) = \tilde{\phi}^t(\tilde{\phi}^t(z)).$

    We consider the smallest $k> 0$ such that $\tilde{\phi}^k(y) = \tilde{\phi}^k(\tilde{\phi}^t(z)).$ Such a $k$ must exist because $k= t$ makes the equality hold, but when $k=0$, the equality does not hold.

    Observe that $\tilde{\phi}^{k-1}(y) \ne \tilde{\phi}^{k-1+t}(z)$, but by ``basic observations", $\text{dist}(\tilde{\phi}^{k-1}(y), \tilde{\phi}^{k-1+t}(z)) \le 2$. By applying the ``basic observation" earlier $k-1$ times, we arrive at $\text{dist}(y,\tilde{\phi}^t(z)) \le 2$. Note that $ \tilde{\phi}^t(z)$ is a vertex of degree less than $i^*$, $y$ is a copy of some vertex, and our construction of $G$ satisfy that vertices that are a copy of some vertex in $H$ must have distance at least $5$ from vertices of degree less than $i^*$, so this is impossible.
\end{proof}

By Claim~\ref{claim:topdeg}, $\deg_H(\phi(v))<i^*$ implies that $v$ is not in one of the top groups, 

\begin{corollary}
    There are at most $N$ vertices $v\in V(G)$ with virtual degree $i^*$ satisfying $\deg_H(\phi(v))<i^*$.
\end{corollary}

We now handle Case II, where $x$ has virtual degree $i$ and $\deg_H(\phi(v))=i^*$. 
\begin{claim} 

There are at most $Nd+N$ vertices $x\in V(G)$ with virtual degree $i^*$ such that $\deg_H(\phi(x)) = i^*$. 
\end{claim}
\begin{proof}
Observe that if $v$ is a top vertex or top neighbor (so we constructed $C_d$ copies of $v$ in $G$), and $x$ is a copy of $v\in V(H)$, then $\deg_{G}(x) = \deg_H(v)$. 

In this case, this means $\deg_{G}(x)=i^*=\deg_H(\phi(x))$. This means there is a bijection $N_{G}(x) \leftrightarrow N_H(\phi(x))$. If $N_{G}(x) = \{z_1,\cdots,z_{i^*}\}$, then $\phi^{-1}(\phi(x))$ can contain only vertices adjacent to all of $z_1,\cdots,z_{i^*}$. However, the only vertices adjacent to all of $z_1,\cdots,z_{i^*}$ that arise from a copy of a top vertex must be the vertices in the twin group of $x$. This shows that $\phi^{-1}(\phi(x))$ contains at most $d$ vertices that are a copy of a top vertex. 

Thus, in Case II, there are at most $Nd$ vertices that are a copy of a top vertex and at most $N$ vertices that are not a copy of a top vertex. 
\end{proof}

Case III can be handled in the following claim. 
\begin{claim}
    There are at most $N$ vertices $x \in V(G)$ of virtual degree $i^*$ such that $\deg_H(\phi(x)) > i^*$. 
\end{claim}
\begin{proof}
Note in Case III, since $\deg_H(\phi(x)) > i^*$, $x$ must not be a copy of some vertex in $H$, so there are at most $N$ vertices of virtual degree $i^*$ in case III. 
\end{proof}

Combining the three cases above, the number of vertices of virtual degree $i$ in the blow-up is at most
$$ N + Nd + N + N = N(d+3).$$

Note that this bound can clearly be improved, but all we care for now is the independence of this bound from $C_d$. Recall $\frac T3 = \frac{N}{6d^{20d-20i+31}} \ge \frac{N}{d^{20d+20}}$ is the number of top blocks. We have $C_d^{T/3}$ naive blow-ups, all of them differ in the set of virtual degree $i$ vertices. We construct an arbitrary map 
$$ f \colon \{ \text{Vertex sets that form naive blow-ups} \} \to  \{ \text{Vertex sets that form maximal blow-ups} \},$$
such that $f(B) \supseteq B$ for all $B$. We showed that each maximal blow-up contains at most $N(d+3)$ vertices in $G$ of virtual degree $i$, and since the naive blow-ups all differ in the set of virtual degree $i$ vertices, every maximal blow-up contains at most $2^{N(d+3)}$ naive blow-ups. This tells us that the number of maximal blow-ups is at least 
$$ \frac{C_d^{N/d^{20d+20}}}{2^{N(d+3)}} = \left( \frac{C_d^{1/d^{20d+20}}}{2^{d+3}}\right)^N.$$

We finally specify $C_d = 2^{d^{21d+20}}$. Then the graph $G$ has at most $NC_d = 2^{d^{21d+20}}N$ vertices, and the maximal $H$-blow-ups in $G$ is at least 
$$ \left( 2^{d^{d}-d-3}\right)^N \ge \left( 2^{d^{d}-d-3}\right)^{v(G)/2^{d^{21d+20}}}.$$
\end{proof}

 Note that this construction and proof works for (induced) copies or with prescriptions.

\begin{theorem}\label{thm: boundedprescribed}
    Let $d\ge 2$ be a positive integer. Let $H$ be a connected graph with maximum degree $d$ and $N$ vertices, where $N > d^{d^{d^d}}$. Let $U\subset V(H)$ be a fixed subset of vertices. Then there exists a graph $G$ such that $N \leq |V(G)| \leq 2^{d^{21d+20}}N$, the maximum degree is at most $d2^{d^{21d+20}}$, and there are at least $2^{|V(G)|/2^{d^{21d+20}}}$ vertex sets in $G$ corresponding to maximal blow-ups of $(H,U)$.   
\end{theorem}

\begin{theorem}\label{thm:boundedinduced}
    Let $d\ge 2$ be a positive integer. Let $H$ be a connected graph with maximum degree $d$ and $N$ vertices, where $N > d^{d^{d^d}}$. Let $U_+,U_-\subset V(H)$ be two disjoint subsets of vertices. Then there exists a graph $G$ such that $N \leq |V(G)| \leq 2^{d^{21d+20}}N$, the maximum degree is at most $d2^{d^{21d+20}}$, and there are at least $2^{|V(G)|/2^{d^{21d+20}}}$ vertex sets in $G$ corresponding to maximal induced blow-ups of $(H,U_+,U_-)$.   
\end{theorem}

\section{Other Examples}\label{sec:unbounded}
When the patterns in an infinite family $\mathcal{H}$ do not have bounded maximum degree, the situation becomes more complex. In some cases, the bound may be polynomial; in others, it may be exponential. We present two such examples in this section.

\begin{proposition}[$F_N$ with no prescriptions and uninduced is polynomial]
    Fix $N>c$. Consider the star $F_N$ with $N$ leaves with no prescriptions. Every $c$-closed graph $G$ on $n$ vertices has $\le 3^cn^2 + n^c$ maximal uninduced blow-ups of $F_N$.
\end{proposition}

\begin{proof}

    For every maximal uninduced $F_N$-blow-up $G[W]$, note that $W$ can be partitioned into $S\sqcup T$ where $|T| \ge N$ and $G[S,T]$ is complete bipartite. Since $N> c$, $G[S]$ is complete. We have two cases:

    \begin{itemize}
        \item $|S| > c$. Then $G[W]=G[S\sqcup T]$ is complete, so this is equivalent to counting maximal cliques of size $\ge N+1$. There are $\le 3^cn^2$ maximal cliques by \cite{paper1}. 
        \item $|S| \le c$. In this case we count the number of pairs $(S,T)$ that make $S\sqcup T$ the vertex set of a maximal $F_N$-blow-up. For every fixed $S$, then if $T' = \bigcap_{s\in S} N(s)$, where $N(s)$ is the neighborhood of $s$ in $G$, then $S\sqcup T'$ is clearly a maximal $F_N$-blow-up. Since $S$ is the set of vertices prescribed onto the unique non-leaf of $F_N$, every $(S,T)$ must satisfy $T\subseteq T'$. Furthermore $T=T'$ because $G[S\sqcup T]$ is a maximal $F_N$-blow-up. Since there are $\le n^c$ choices for $S$, there are $\le n^c$ choices for $(S,T)$.
    \end{itemize}
    The conclusion follows.
\end{proof}

\begin{corollary}
    For any $c>1$, the family $\{(F_N,\emptyset)\}_{N>c}$ is a $c$-closed polynomial pattern family.
\end{corollary}

Note that $\{(F_N,\emptyset)\}_{N\ge 1}$ is also $c$-closed polynomial pattern family, since the number of maximal $F_j$-blow-ups in a graph $G$ is polynomial in $v(G)$ by Theorem \ref{thm:noninduced}.

\begin{lem}
  Fix $h \geq 3$.  Let $\mathcal{F} = \{(F_m,U_m)\}_{m\ge 2}$ where $F_m$ is a tree with depth $h$, and all vertices in depths $0,\ldots,h-1$ have $m$ children, and $U_m$ is an arbitrary subset of vertices of $F_m$. We call $F_m$ the $m$-ary tree with depth $h$. There exists arbitrarily large $n$ such that there exists 2-closed graphs $G_n$ on $n$ vertices such that the number of maximal $(F_m,U_m)$-blow-ups of some $F_m \in \mathcal{F}$ in $G_n$ is at least $\exp(Cn^c)$ where $C,c>0$ are constants depending only on $h$. 
 
\end{lem}

\begin{proof}
    Let $G=F_{2m}$ with root $r$. Note that $v(G)=\sum_{i=0}^h (2m)^i<2(2m)^h.$  We want to count the number of maximal uninduced blow-ups of $F_m$ in $G$. 
    
    We consider blow-ups of $F_m$ of this form: we first take an induced subgraph of the $G$ that is still rooted at $r$ and is $m$-ary subtree of depth $h-1$. Then for each vertex of depth $h$ in $G$ whose parent is in the subtree, we include it in the blow-up. Call the blow-up $A$.

    We claim such a blow-up are necessarily maximal. Let $B$ the a maximal blow-up containing $A$. Let the blow-up morphism from $B$ to $F_m$ be called $\phi$. The key observation is that, if $v\in B$ is not in depth $h$ (i.e. not a leaf in $G$), then $\phi(v)$ must not be a leaf in $F_m$. This is because the longest path starting from such a $v$ in $B$ has length at most $2h-1$, while the longest path starting from any leaf in $F_m$ has length $2h$. An implication is that if $v\in B$ is not in depth $h$, then the fiber of $\phi(v)$ has size $1$, for if the fiber of $\phi(v)$ has size larger than $1$ then $\phi(v)$ has at most one neighbour and thus is a leaf in $F_m$.
    Therefore, the total number of vertices in depth $0$ through $h-1$ of $A,B,$ and $F_m$ are the same. Since $A\subset B$, $B\backslash A$ can only contains vertices of depth $h$. But no such vertex exists because all its neighbours (at depth $h-1$) in $G$ are not contained in $B$. 
    
    Now we count. Note the number of these blow-ups is equal to the number of $m$-ary trees of height $h-1$ in a $2m$-ary tree. At depth $1$, $m$ of the $2m$ vertices must be chosen, so there are $\binom{2m}{m}$ chocies for the vertices of depth $1$. Then at depth $2$, each of these $m$ vertices has $2m$ children of which $m$ must be chosen, so for each of these $m$ vertices there are $\binom{2m}{m}$ choices. Thus $ \binom{2m}{m}^m$ choices for the vertices with depth $2$. Analogously, there are $\binom{2m}{m}^{m^{i-1}}$ choices for the vertices with depth $i$ for all $i=1,\cdots,h-1$, so the total number of choices are $\prod_{i=1}^{h-1}  \binom{2m}{m}^{m^{i-1}} \ge 2^{m+\cdots+m^{h-1}} \ge 2^{m^{h-1}}$  Since $v(G) \le 2(2m)^h$, this shows that the number of maximal blow-ups is at least $2^{2^{-h}v(G)^{\frac{h-1}{h}}}$.  
   
\end{proof}

Lastly, we prove another general result regarding enumerating maximal blow-ups of patterns in an infinite family.  
\begin{proposition}\label{prop:codeg}

    Let $H$ be a $c$-closed graph with the minimum degree $\delta(H)$ is relatively compared to the the maximum codegree, i.e. $\delta(H) > 2+2 \max\limits_{u\ne v \in V(H)} codeg(u,v)$. Furthermore, $ \max\limits_{u\ne v \in V(H)} codeg(u,v) < c$ and $\delta(H)\geq 5$. Then there exists a $2c$-closed graph  $G$ on $2v(H)$ vertices such that the number of maximal uninduced $H$-blow-ups (with no prescriptions) in $G$ is at least $\exp(Cv(H)^{1/2})$ for some universal constant $C>0$. 
\end{proposition}

Such a graph exists. A randomized construction is given by $H\sim G(n^{0.99}, n, n^{-0.98})$ where $n\to \infty$; this is a random bipartite graph on $A\sqcup B$ where $|A| = n^{0.99}, |B|=n$ and each edge $(a,b) \in A\times B$ is constructed with probability $n^{-0.98}$, and this graph has maximal codegree at most $3$ (which implies it is $4$-closed) with high probability.  Thus, we only need $\delta(G)> 2+2\cdot 3=8.$ The probability that $\delta(G)\leq 8$ is less than $2n\begin{pmatrix}n\\n-8\end{pmatrix}(1-n^{-0.98})^{n-8}$, which goes to zero as $n$ approaches infinity. Hence, when $n$ is large, we almost surely get a graph with all properties of $H$ specified in~\ref{prop:codeg}.

Proposition \ref{prop:codeg} leads to the following corollary regarding an infinite family with large minimum degree with respect to the maximum codegree. 
\begin{corollary}
    Fix $c > 0$. Let $\mathcal{H} = \{(H_1, U_1), (H_2, U_2),\dots\}$ be an infinite family of graphs. Assume that for each $H_i$, $H_i$ is $c/2$-closed with $\delta(H_i) > 2+2 \max\limits_{u\ne v \in V(H_i)} codeg(u,v)$. Furthermore, $ \max\limits_{u\ne v \in V(H_i)} codeg(u,v) < c$ and $\delta(H_i)\geq 5$. Then $\mathcal{H}$ is a $c$-closed exponential family. 
\end{corollary}

\begin{proof}[Proof of Proposition \ref{prop:codeg}]
   We find an acyclic subgraph $T$ of $H$ such that every vertex of $H$ is incident to an edge in $T$. Let $G$ be our new graph on $2v(H)$ vertices. For every vertex $v\in H$, we construct $v^{(1)}, v^{(2)} \in V(G)$.   

    \begin{itemize}
        \item If $vw \notin E(H)$, then $v^{(i)} w^{(j)} \notin E(G)$ for all $i,j=1,2$
        \item If $vw \in E(H) \setminus T$ then $v^{(i)}w^{(j)} \in E(G)$ for all $i,j$. 
        \item If $vw \in T$, then $v^{(i)}w^{(j)}\in E(G)$ if and only if $i=j$. 
        \item For all $v\in V(H)$, $v^{(1)}v^{(2)} \in E(G)$.
    \end{itemize}

    Note $G$ is $(2c+1)$-closed: choose arbitrary non-adjacent $a^{(i)}, b^{(j)} \in V(G)$, where $a,b\in V(H)$ and $i,j\in \{1,2\}$. If $x^{(k)}$ is adjacent to both $ a^{(i)}, b^{(j)}$ in $G$, then it must follow that either $x=a^{(3-i)}$ or $x=b^{(3-j)}$ or $x$ is adjacent to both $ a,b$ in $H$. There are at most $c-1$ such choices for $x$, so there are at most $2c$ choices for $x^{(k)}$.

    The key lemma is as follows:
    
    \begin{lem}
        An $H$-blow-up in $G$ that contains exactly one of $v^{(1)}, v^{(2)}$ for each $v\in H$ is a maximal $H$-blow-up.  
    \end{lem}
    
    \begin{proof}
        Let $A \subset V(G)$ be our blow-up, where $A = \bigsqcup_{v\in V(H)} A_v$ and $A_v$ is the part corresponding to the vertex $v\in V(H)$ \footnote{for $v\in V(H)$, it is not a priori true that $A_v \subseteq \{v^{(1)}, v^{(2)}\}$, but we do know if $v\sim w$ then $G[A_v, A_w]$ is a complete bipartite graph}. Suppose for each $v\in V(G)$, there exists $j\in \{1,2\}$ such that $v^{(j)} \in A$.
    \begin{observation}\label{obs:code}
    If $v\ne w$ then for all $i,j \in \{1,2\}$, if $v^{(i)}, w^{(j)} \in A$, then they must belong to different parts of a blow-up. In other words, for all $v\in V(H)$, there exists $x\in V(H)$ such that $A_v \subseteq \{x^{(1)},x^{(2)}\}$. 
    \end{observation}
    This holds because if $v^{(i)}, w^{(j)}$ belong to the same part of the blow-up, they must have at least $\delta(H)$ common neighbors. But the number of common neighbors they have in $G$ is at most $ 2 codeg_H(v,w)+2 < \delta(H)$, Contradiction. 
    
    Thus, by Observation~\ref{obs:code}, there exists a mapping $\pi \colon V(H) \to V(H)$ such that $\emptyset \ne A_v \subseteq \{ \pi(v)^{(1)}, \pi(v)^{(2)}\}$ for all $v\in V(H)$. Since for each $v$, there exists $w$ such that $A_w \cap \{v^{(1)}, v^{(2)}\} \ne \emptyset$,  $\pi$ is surjective, hence bijective. This in particular also implies that if $v^{(1)}, v^{(2)}\in A$, then there exists $w$ such that $A_w = \{ v^{(1)},v^{(2)}\}$. 
    
    Note that $\deg_H(\pi(v)) \ge \deg_H(v) $ for all $v\in V(H)$ because there exists $j\in \{1,2\}$ such that $\pi(v)^{(j)}\in A_v$, and $\pi(v)^{(j)}$ needs to be adjacent to $\deg_H(v)$ vertices in $G$ from distinct parts of the blow-up, which corresponds to $\pi(v)$ being adjacent to $\deg_H(v)$ distinct vertices in $H$. Since $\sum_{v\in V(H)} \deg(\pi(v)) = \sum_{v\in V(H)} \deg(v)$, it follows that $\deg(\pi(v)) = \deg(v)$ for all $v\in V(H)$.

    Now we will show that $|A_v|=1$ for all $v$. Assume not, then $A_v = \{ \pi(v)^{(1)}, \pi(v)^{(2)}\}$ for some $v$. Note that $\pi(v)$ is incident to an edge in $T$, say $\pi(v)w$. Then $\pi(v)^{(1)}, \pi(v)^{(2)}$ does not have a common neighbor in $\{w^{(1)}, w^{(2)}\} \supset A_{\pi^{-1}(w)}$. So $A_v$ has common neighbors in $<\deg(\pi(v)) = \deg(v)$ parts of the blow-up, which is a contradiction. This proves our lemma. 
    \end{proof}

This shows that in particular, if the following two conditions holds:
\begin{enumerate}
\item $A
_v \in \{\{v^{(1)}\},\{v^{(2)}\}\}$ for all $v\in V(H)$;
\item if $vw\in E(T)$ then $A_v = \{v^{(1)}\} \iff A_w = \{w^{(1)}\}$,
\end{enumerate}
then $\bigsqcup\limits_{v\in V(H)} A_v$ is a maximal $H$-blow-up. The number of such maximal blow-ups is at least $ 2^{v(H) - e(T)}$, because we can first choose $A_v$ from $\{\{v^{(1)}\},\{v^{(2)}\}\}$ for an arbitrary $v\in V(H)$, and for each $w$ incident to $v$ in $T$, $A_w$ is determined by $A_v$. This tells us that we can choose $v(H) - e(T)$ times, so the number of blow-ups \emph{of that form} is at least $2^{v(H)-e(T)}$.

We now want to compute $\max_T (v(H)-e(T))$ where every vertex of $H$ is incident to an edge in $T$.

\begin{lem}
    Suppose $H$ is a graph with no isolated vertices. Then $\max_T (v(H)-e(T)) $ is equal to the number of edges in a maximal matching in $H$.
\end{lem}

\begin{proof}
    Consider an arbitrary acyclic subgraph $T$ of $H$ such that every vertex of $H$ is incident to an edge in $T$. Let $M$ be a maximal matching of $T$. We claim that $v(H)-e(T)\leq e(M)$. This is because every vertex in $V(H)\setminus V(M)$ is incident to $M$ in $T$, so $e(T)\geq e(M)+|V(H)\setminus V(M)|=v(H)-e(M).$

    On the other hand, given any maximal matching $M$ of $H$, we can construct an acyclic subgraph $T$ of $H$ such that every vertex of $H$ is incident to an edge in $T$. First we put $M$ into $T$. Then, for each $v\in V(H)\setminus V(M)$, we pick an arbitrary edge between $v$ and $M$, and add this edge to $T$. 

     The conclusion readily follows.
\end{proof}

Let $M$ be a matching using edges in $H$. Note that $M$ is a maximal matching if and only if $H[V(H) \setminus V(M)]$ is an independent set. Thus, $|V(H) \setminus V(M)| \le \alpha(H) $, where $\alpha(H)$ is the size of the largest independent set in $H$. Thus, if $M$ is a maximal matching, then $ e(M) = \frac{v(M)}{2} \ge \frac{v(H)-\alpha(H)}{2}$. 

\begin{lem}
    Suppose $H$ is a graph where the minimum degree $\delta(H)$ is greater than twice the maximum codegree, and $\delta(H) \ge 5$. Then  $\alpha(H) \le n-\frac 12 n^{1/2}$, where $n = v(H)$. 
\end{lem}

\begin{proof}
    Let $V=V(H)$. Assume there exists an independent set with greater than $n-\frac 12 n^{1/2}$ vertices, call $I$. Then 

    \begin{itemize}
        \item $ e(V\setminus I) + e(V\setminus I,I) = e(H)\ge \frac n2 \delta (H)$
        \item Since $|V\setminus I | \le \frac 12 n^{1/2}$, it follows that $e(V\setminus I) \le n/8$. Thus, $e(V\setminus I,I) \ge \frac n2 \delta(H) - \frac n8.$
       
    \end{itemize}
     We compute the number of triplets $(v_1,i,v_2)$ where $i\in I$ and $v_1,v_2\in N(i), v_1\neq v_2$ in two ways: 

     On one hand, if we first fix $i$ and then choose $(v_1,v_2)$, we can see it is equal to $\sum_{i\in I} \deg(i)^2 - \deg(i) \ge \frac{(\sum_{i\in I} \deg(i))^2}{|I|} - (\sum_{i\in  I} \deg(i)) = \frac{e(V\setminus I,I)^2}{|I|} - e(V\setminus I,I) \ge n(\frac 12 \delta(H)-\frac 18)(\frac 12 \delta(H)-\frac 98)$, since $e(V\setminus I,I) \ge \frac n2(\delta (H)-\frac 14)$.

     On the other hand, it is equal to $$\sum_{(v_1,v_2)\in V\setminus I, v_1\ne v_2} codeg(v_1,v_2) \le |V\setminus I|^2 \max_{u,v\in V(H)} codeg(u,v) \le \frac n4  \frac{\delta(H)}{2}.$$ Since $\delta(H)\ge 5$, these two inequalities contradict each other, so the lemma follows.

\end{proof}
The proposition follows: $$2^{v(H)-e(T)}\geq 2^{e(M)}\geq 2^{\frac{v(H)-\alpha(H)}{2}}\geq 2^{\frac 12 n^{1/2}}.$$
\end{proof}

\section{Open Questions}

We show that for any constant $c > 0$ and any fixed graph $H$, the number of maximal (not necessarily induced) $H$-blow-ups in a $c$-closed graph $G$ is always bounded by a polynomial in the number of vertices of $G$. For induced blow-ups, we provide a complete characterization of the graphs $H$ for which the number of maximal induced $H$-blow-ups remains polynomial, and also show a sharp polynomial-versus-exponential dichotomy. We also extend our investigation to blow-ups of graphs drawn from infinite families, and prove several general results. In particular, we show that if the family consists of bounded-degree graphs, then the number of maximal blow-ups from the family in a $c$-closed graph can be exponential.

As we have seen earlier, when the patterns in an infinite family $\mathcal{H}$ have unbounded maximum degree, it becomes unclear what the correct bound should be. In particular, does there exist an non-trivial infinite family that is neither $c$-closed polynomial nor $c$-closed exponential for any fixed $c$? If every such family must fall into one of these two categories, what would a characterization theorem look like for the infinite-family case?

 In section~\ref{sec: noninduced finite} and section~\ref{sec: induced finite}, we gave a full characterization of when the number of maximal induced (non-induced) blow-ups of a single graph $H$ is polynomial (exponential) in $c$-closed graphs. This result can be generalized to count maximal blow-ups of a finite pattern family. Thus when studying the number of maximal blow-ups in infinite pattern families, it will be interested to focus on the ``truly infinite" pattern families, (or in some sense a ``minimal family") as defined below. However, it may be difficult in general to determine whether a pattern family is truly infinite.

 \begin{definition}[Truly infinite pattern family]
Let $\mathcal{H} = \{(H_1,U_1), (H_2,U_2), \dots\}$ be a family of patterns, where $H_i$ is a simple graph and $U_i \subseteq V(H_i)$ is the set of vertices in $H_i$ that are prescribed to be cliques. We say $\mathcal{H}$ is \textbf{truly infinite}, if there does not exists a finite subset $\mathcal{H'}\subsetneq \mathcal{H}$ such that the following holds: $\forall S\subset V(G)$, $\forall (H_i,U_i)\subset \mathcal{H}$, if $G[S]$ is a maximal blow-up of $H_i$, then there exists $(H_j,U_j)\in \mathcal{H'}$ such that $G[S]$ is a maximal blow-up of $H_j$.
\end{definition}

Definition~\ref{def:maxblow-up} of maximal blow-ups on a single graph can be generalized to maximal blow-ups on a family of graphs. There are potential research interest in this object.

\begin{definition}\label{def:maxblow-upfamily}
Let $\mathcal H=\{H_1,H_2,\ldots\}$ be a family of simple graphs. Let $G$ be a simple graph and $S\subset V(G)$. We call $G[S]$ a \textbf{maximal blow-up} of $\mathcal H$ in $G$, if $G[S]$ is a blow-up of some $H_i\in \mathcal H$ and there does not exist $S\subsetneq S'\subset V(G)$ and $H_j\in \mathcal H$ such that $G[S']$ is a blow-up of $H_j$.
\end{definition}

Does there exist an infinite family $\{ (H_i,U_{+,i}, U_{-,i})\}_{i\ge 1}$ that fails to form an induced $c$-closed exponential pattern family, i.e. there does not exist $\epsilon > 0$ and a sequence of $c$-closed graphs $(G_n)_{n\ge 1}$ with $v(G_n)\to \infty$, and for all $n$, 
\[ \# \{S \subset V(G_n) \colon \text{ there exists } i \text{ such that } G_n[S] \text{ is a maximal } H_i\text{-blow-up}\} \ge e^{v(G_n)^{\epsilon}} ?\]

The question is equivalent to: for every $\epsilon>0$, does there exists an $N = N(\epsilon)$ such that for all graphs $G$ with $n \ge N$ vertices, 
\[ \# \{S \subset V(G) \colon \text{ there exists } i \text{ such that } G_n[S] \text{ is a maximal } H_i\text{-blow-up}\} < e^{n^{\epsilon}}?\]

\bibliographystyle{abbrv} 
\bibliography{writeup} 
\end{document}